\documentclass[10pt, reqno]{amsart}
\usepackage[utf8]{inputenc}

\usepackage{style}
\usepackage{needspace}
\usepackage{dsfont}
\newcommand{\I}{\mathds{1}}

\title{Random independent sets in uncrowded hypergraphs}

\author{Abhishek Dhawan$^\star$}
\address{$^\star$Department of Mathematics, University of Illinois Urbana--Champaign}
\email{adhawan2@illinois.edu}
\thanks{Abhishek Dhawan is supported by NSF RTG grant DMS-1937241.}

\author{Lina Li$^\dagger$}
\address{$^\dagger$Department of Mathematics, University of Mississippi}
\email{lli5@olemiss.edu}
\thanks{Lina Li is supported by the Simons Foundation SFI-MPS-TSM-00024322.}

\author{Abhishek Methuku$^\star$}
\email{methuku@illinois.edu}
\thanks{Abhishek Methuku is supported by the UIUC Campus Research Board Award RB25050 and a Simons award SFI-MPS-TSM-00025583.}

\author{Minh-Quan Vo$^\star$}
\email{vo15@illinois.edu}

\author{Kewen Yuan$^\star$}
\email{kyuan8@illinois.edu}

\date{}

\begin{document}

\begin{abstract}
Given any fixed integer $k \ge 2$ and sufficiently large $d$, we show that the largest possible fractional chromatic number of a $k$-uniform $d$-degenerate uncrowded hypergraph $\mathcal{H}$ (i.e., with girth at least $5$) satisfies
\[
\chi_f(\mathcal{H}) = (1 + o_d(1)) \left((k-1)\,\frac{d}{\log d}\right)^{\frac{1}{k-1}}.
\]
In fact, we prove that this holds for $k$-uniform $d$-degenerate hypergraphs of girth at least $g$, for any given $g \ge 5$. As a corollary, we obtain improved bounds on the fractional chromatic number of $d$-degenerate linear hypergraphs. This work builds upon a recent result by Allen, Dhawan, and Noel, extending it from graphs to hypergraphs. In addition to overcoming the new difficulties that arise in the hypergraph setting, our approach yields a simpler proof even in the original graph case.

Our proof of the upper bound uses a simpler iterative procedure for sampling independent sets. We also establish bounds for fractional colorings with local demands, a framework introduced by Kelly and Postle, verifying a recent conjecture of Yu and Zhang. As a consequence, we obtain a degree-sequence bound on the independence number of uncrowded hypergraphs with a leading constant matching the shattering threshold.

For the matching lower bound, we use a hypergraph variant of the uniform attachment model and harmonic vertex weights to bound the fractional chromatic number via linear programming duality, then remove all short cycles by deleting vertices of negligible total weight. We also replace the Catalan-number argument used in the graph case with a matrix-norm estimate, simplifying the analysis.
\end{abstract}

\maketitle

\section{Introduction}\label{section: intro}

\subsection{Background}\label{subsection: background}

A \emph{hypergraph} is an ordered pair $\cH=(V,E)$, where $V$ is a finite set of \emph{vertices} and $E\subseteq 2^V$ is a family of \emph{edges}. We say that $\cH$ is \emph{$k$-uniform} (or a \emph{$k$-graph}) if every edge contains precisely $k$ vertices; under this definition, an ordinary graph is simply a $2$-uniform hypergraph. A set $I\subseteq V$ is \emph{independent} if it contains no edge of $\cH$, and the \emph{independence number} $\alpha(\cH)$ is the maximum cardinality of an independent set in $\cH$. A \emph{proper coloring} of $\cH$ assigns colors to the vertices in $V$ such that no edge is monochromatic, and the \emph{chromatic number} $\chi(\cH)$ is the minimum number of colors required for a proper coloring.

It is well known that every $n$-vertex graph $G$ with maximum degree $\Delta$ satisfies
\begin{equation}\label{eq:greedy}
    \alpha(G)\ge \frac{n}{\Delta+1} \qquad\text{and}\qquad \chi(G)\leq \Delta+1.
\end{equation}
These basic inequalities are commonly referred to as the \emph{greedy bounds}. A central theme in extremal graph and hypergraph theory is determining when local sparsity guarantees a substantial improvement over the greedy coloring bound. At the level of independent sets, this direction originates in the foundational work of Ajtai, Koml\'os, and Szemer\'edi~\cite{ajtai1980note, ajtai1981dense}, who showed that triangle-free graphs contain independent sets of size $\Omega\left(\frac{n\log\Delta}{\Delta}\right)$. Shearer~\cite{shearer1983note, shearer1991note} subsequently sharpened the leading constant to $1 - o(1)$, which remains the state of the art for this problem.

\Needspace{6\baselineskip}
In 1995, Kim~\cite{Kim95} established that every graph of girth at least $5$ and maximum degree at most $\Delta$ satisfies
\begin{equation}\label{eq:kim}
    \chi(G) \leq (1+o_\Delta(1)) \,  \frac{\Delta}{\log \Delta}.
\end{equation}
Independently, Johansson~\cite{Joh_triangle} proved that triangle-free graphs with maximum degree at most $\Delta$ satisfy $\chi(G) = O(\Delta/\log\Delta)$.\footnote{While commonly cited as a DIMACS Technical Report, no record of Johansson's manuscript exists. Instead, Molloy and Reed~\cite{MolloyReed} present a variant of Johansson's proof demonstrating that $\chi(G) \leq 160\frac{\Delta}{\log \Delta}$ for triangle-free graphs $G$.} Molloy~\cite{Molloy} later extended Kim's bound~\eqref{eq:kim} to all triangle-free graphs; see~\cite{bernshteyn2019johansson, hurley2021first, DKPS, bonamy2022bounding, martinsson2021simplified} for alternative proofs. More broadly, coloring under local sparsity constraints---where the neighborhood of every vertex spans few edges---has developed into an extensive area of research~\cite{AKSConjecture, vu2002general, DKPS, hurley2021first, anderson2024coloring, anderson2025coloring, bradshaw2025toward}; see also the survey~\cite{KangKelly2023nibble} and the references therein.

A graph $G$ is \emph{$d$-degenerate} if every subgraph of $G$ contains a vertex of degree at most $d$. Because one may replace the maximum degree $\Delta$ with degeneracy $d$ in the greedy bound in~\eqref{eq:greedy}, it is natural to ask whether Kim's bound~\eqref{eq:kim} or Johansson's $O(\Delta/\log\Delta)$ bound holds with maximum degree $\Delta$ replaced by degeneracy $d$. For the ordinary chromatic number, this fails without additional constraints on the vertex count: as observed by Alon, Krivelevich, and Sudakov~\cite{AKSConjecture}, a variant of Tutte's construction yields triangle-free $d$-degenerate graphs with chromatic number $d+1$. However, these counterexamples are necessarily large. Brada\v{c}, Fox, Steiner, Sudakov, and Zhang~\cite[Theorem~1.1]{bradavc2026coloring} recently proved that there are absolute constants $c,C>0$ such that every $n$-vertex $d$-degenerate triangle-free graph $G$ with $n\leq 2^{cd}$ satisfies $\chi(G)\leq Cd/\log(d/\log n)$. In particular, for every fixed $\eta\in(0,1)$, this gives $\chi(G) \le C d/\log d$ for some constant $C > 0$ depending on $\eta$, given that $n\leq\exp(d^{1-\eta})$ and $d$ is sufficiently large. Thus, degeneracy behaves quite differently from maximum degree in ordinary graph coloring.

The situation is different for the \emph{fractional chromatic number} $\chi_f(G)$, defined as the optimal value of the following linear programming relaxation of graph coloring:
\begin{mini}
    {\lambda \in \R^{|\mathcal{I}(G)|}}{\sum_{I \in \mathcal{I}(G)}\lambda_I}{\label{eq:chif}}{}
    \addConstraint{\sum_{\substack{I\in\mathcal{I}(G)\\v\in I}}\lambda_I}{\ge 1}{\qquad v \in V(G)}
    \addConstraint{\lambda_I}{\ge 0}{\qquad I \in \mathcal{I}(G)},
\end{mini}
where $\mathcal{I}(G)$ denotes the family of all independent sets of $G$. As usual, any proper coloring of $G$ corresponds to an integer feasible solution to this program, giving $\chi_f(G)\le\chi(G)$. Furthermore, any feasible solution $\lambda$ to~\eqref{eq:chif} with value $q$ induces a probability distribution $\mathcal{D}$ over $\mathcal{I}(G)$ such that $\Pr_{I \sim \mathcal{D}}(v\in I) \ge 1/q$, obtained by selecting an independent set $I$ with probability $\lambda_I/q$. The starting point of our investigation is the following recent result of Martinsson and Steiner~\cite{martinsson2025random}, which resolved a conjecture of Harris~\cite{harris2019some}:

\begin{theorem}[{Martinsson--Steiner~\cite{martinsson2025random}}]\label{theo: martinsson steiner}
Every $d$-degenerate triangle-free graph $G$ satisfies 
\[
\chi_f(G) \le (4 + o_d(1))\,\frac{d}{\log d}.
\]
\end{theorem}

Martinsson and Steiner~\cite[Conjecture~5.2]{martinsson2025random} conjectured that every $d$-degenerate triangle-free graph $G$ satisfies $\chi_f(G)\le(1+o(1))\frac{d}{\log d}$, and that this bound is asymptotically sharp. Allen, Dhawan, and Noel~\cite{high_girth} established the conjectured upper bound under the additional assumption that the girth is at least $5$. Moreover, they established the conjectured lower bound in a stronger form, for graphs of arbitrarily large girth, as follows.

\begin{theorem}[{Allen--Dhawan--Noel~\cite{high_girth}}]\label{theo: girth}
    The following statements hold for every fixed $g \ge 4$ and all sufficiently large $d$.
    \begin{enumerate}[label=\ep{\roman*}]
        \item\label{item: high girth ub graph} Every $d$-degenerate graph $G$ of girth at least $5$ satisfies $\chi_f(G)\le(1+o_d(1))\frac{d}{\log d}$.
        \item\label{item: high girth lb graph} There exists a $d$-degenerate graph $G$ of girth at least $g$ satisfying $\chi_f(G)\ge(1-o_d(1))\frac{d}{\log d}$.
    \end{enumerate}
\end{theorem}

In this paper, we extend both parts of Theorem~\ref{theo: girth} to
$k$-uniform hypergraphs for all $k \ge 3$. Our proofs build on the methods of Allen, Dhawan, and Noel~\cite{high_girth}: for the upper bound, we introduce a simpler iterative procedure for sampling independent sets; for the lower bound, we simplify the short-cycle analysis using a matrix-norm estimate.

For $\ell\ge 2$, an \emph{$\ell$-cycle} is a hypergraph with edge set $\{e_0,e_1,\dots,e_{\ell-1}\}$ containing distinct \emph{core vertices} $v_1,\dots,v_\ell$ such that $v_i\in e_i\cap e_{i+1}$ for every $1 \le i \le \ell$, with subscripts taken modulo~$\ell$. The \emph{girth} of a hypergraph is the length of its shortest cycle. A hypergraph is \emph{linear} if it contains no $2$-cycles, and, following Ajtai et al.~\cite{AKPSS1982}, it is called \emph{uncrowded} if it has girth at least~$5$, or equivalently, if it contains no $\ell$-cycles for $2\leq \ell\leq 4$.

The Lov\'asz Local Lemma implies that every $k$-uniform hypergraph $\cH$ of maximum degree $\Delta$ satisfies $\chi(\cH) = O\left(\Delta^{\frac{1}{k-1}}\right)$. A large body of literature seeks to improve this local lemma bound under structural constraints.
In 2013, Frieze and Mubayi~\cite{frieze2013coloring} extended Johansson's bound on triangle-free graphs to hypergraphs, proving that every $k$-uniform hypergraph $\cH$ with maximum degree at most $\Delta$ and girth at least $4$ satisfies
\begin{equation}\label{eq:FM}
\chi(\cH) \leq (c + o_\Delta(1)) {\left(\frac{\Delta}{\log \Delta}\right)^{\frac{1}{k-1}}},
\end{equation}
for a constant $c > 0$ depending only on $k$. Analogous bounds in the broader setting of triangle-free hypergraphs were obtained by Cooper and Mubayi~\cite{cooper2015list} for $3$-graphs and by Li and Postle~\cite{lipostle2022} for arbitrary uniformities $k\ge 3$. Here, a \emph{triangle} is a collection of three distinct edges $e,f,g$ and three distinct vertices $u,v,w$ such that $u\in e\cap f$, $v\in f\cap g$, $w\in e\cap g$, and $\{u, v, w\}\cap e \cap f \cap g = \varnothing$.

While the order of magnitude in~\eqref{eq:FM} is optimal (see, e.g.,~\cite{bennett2024chromatic}), determining the optimal leading constant has remained a major open problem, even for uncrowded hypergraphs. In 2021, Iliopoulos~\cite{iliopoulos2021improved} showed that one can take $c=k-1$ in~\eqref{eq:FM}, extending Kim's bound~\eqref{eq:kim} to all $k$-uniform uncrowded hypergraphs for every $k \ge 3$. Very recently, Dhawan, Methuku, and Vo~\cite{DMV26coloring} and, independently, Yu and Zhang~\cite{yu2026list} strengthened this bound by improving the leading constant to $c = (k-1)^{\frac{1}{k-1}}$, matching the well-studied \textit{shattering threshold} for independent sets and colorings of sparse random hypergraphs~\cite{Achlioptas, Zdeborova, ayre2019hypergraph, ayre2019rigid}. More specifically, their results imply the following.

\begin{theorem}[{Dhawan--Methuku--Vo~\cite{DMV26coloring}; Yu--Zhang~\cite{yu2026list}}]\label{theo: chromatic number bound}
Let $k \ge 2$ be an integer. Every $k$-uniform uncrowded hypergraph $\cH$ of maximum degree at most $\Delta$ satisfies 
\[
\chi(\cH) \le (1+o_\Delta(1)) {\left((k-1) \, \frac{\Delta}{\log\Delta}\right)^{\frac{1}{k-1}}}.
\]
\end{theorem}

We note that replacing maximum degree $\Delta$ with degeneracy $d$ causes the local lemma bound on $\chi(\cH)$ to fail, as Wood~\cite{wood2014hypergraph} constructed $k$-uniform $d$-degenerate hypergraphs with $\chi(\cH) = d + 1$. For fractional colorings, however, a simple application of the alteration method yields the bound $\chi_f(\cH) \le \left(\frac{k}{k-1}\right)\left(kd\right)^{\frac{1}{k-1}}$~\cite[Fact~1]{dhawan2026fractional}. Recently, Dhawan~\cite{dhawan2026fractional} showed that this bound can be substantially improved under local sparsity constraints. Focusing on hypergraphs of girth at least $4$, he extended the bound~\eqref{eq:FM} of Frieze and Mubayi by replacing the maximum degree with the degeneracy:

\begin{theorem}[{Dhawan~\cite{dhawan2026fractional}}]\label{theo: hypergraph girth 4}
Let $k \ge 2$ be an integer.
Every $k$-uniform $d$-degenerate hypergraph $\cH$ of girth at least $4$ satisfies
\[
\chi_f(\cH) \le (1+o_d(1))\left(\frac{k}{k-1}\right)\left(k(k-1)^2\,\frac{d}{\log d}\right)^{\frac{1}{k-1}}.
\]
\end{theorem}

The bound in Theorem~\ref{theo: hypergraph girth 4} has the correct order of magnitude, but is not asymptotically optimal, even when $k = 2$. For every fixed integer $k \ge 2$, we establish an asymptotically optimal bound on the fractional chromatic number of $k$-uniform uncrowded hypergraphs in terms of their degeneracy. Our main result, Theorem~\ref{thm: degenerate}, extends Theorem~\ref{theo: girth} from graphs to hypergraphs. Furthermore, we establish a local version in which the requirement at each vertex depends directly on its individual degree, in the spirit of the fractional coloring framework with local demands introduced by Kelly and Postle~\cite{kelly2024fractional}; see Theorem~\ref{thm: local demand}.

\Needspace{6\baselineskip}
\subsection{Main results}\label{subsection: main results}

In this section, we formally state our main results.

\subsubsection*{Uncrowded hypergraphs.}
For every fixed integer $k \ge 2$, we asymptotically 
determine the largest possible fractional chromatic number of a $k$-uniform $d$-degenerate uncrowded hypergraph as $d\to\infty$.

\Needspace{7\baselineskip}
\begin{theorem}\label{thm: degenerate}
Let $k \ge 2$ be an integer, and let $\eps>0$ be arbitrary. The following statements hold.
\begin{enumerate}[label=\ep{\roman*}]
\item\label{item: degenerate ub} There exists $d_0 \in \N$ such that, for every $d \ge d_0$ and $n \in \N$, every $n$-vertex $k$-uniform $d$-degenerate uncrowded hypergraph $\cH$ satisfies
    \[
    \chi_f(\cH) \le (1+\eps) {\left((k-1)\, \frac{d}{\log d}\right)^{\frac{1}{k-1}}}.
    \]
    Moreover, there is a $\mathrm{poly}(n, d)$-time algorithm that samples an independent set $I \subseteq V(\cH)$ such that, for each $v \in V(\cH)$,
    \[
    \pr(v\in I) \ge (1-\eps) {\left(\frac{1}{k-1} \, \frac{\log d}{d}\right)^{\frac{1}{k-1}}}.
    \]

\item\label{item: degenerate lb} For every $g \ge 2$, there exists $d_0 \in \N$ such that the following holds for all $d \ge d_0$. There exists $n_0 \in \N$ such that for every $n \ge n_0$, there exists an $n$-vertex $k$-uniform $d$-degenerate hypergraph $\cH$ of girth at least $g$ satisfying
    \[
    \chi_f(\cH) \ge \left(1-\eps\right)\left((k-1) \, \frac{d}{\log d}\right)^{\frac{1}{k-1}}.
    \]
\end{enumerate}
\end{theorem}

We note that Theorem~\ref{thm: degenerate} shows that
\[
\chi_f(\mathcal{H}) = (1+o(1)) \left((k-1)\frac{d}{\log d}\right)^{\frac{1}{k-1}}
\]
holds even for hypergraphs $\mathcal{H}$ of arbitrarily large fixed girth.

For the upper bound, our proof draws on two earlier arguments: the random independent set procedure from Dhawan's proof of Theorem~\ref{theo: hypergraph girth 4}~\cite{dhawan2026fractional}, and the second-moment analysis from Allen, Dhawan, and Noel's proof of Theorem~\ref{theo: girth}\ref{item: high girth ub graph}~\cite{high_girth}. However, combining these approaches directly does not give an asymptotically optimal bound when $k \ge 3$. We therefore introduce a new random sampling procedure. It is simpler than the one in~\cite{dhawan2026fractional} and yields an asymptotically optimal bound. We give an overview of the proof in Section~\ref{subsection: upper bound overview}.

The argument underlying the proof of Theorem~\ref{thm: degenerate}\ref{item: degenerate ub} extends beyond uncrowded hypergraphs. More precisely, with straightforward modifications, it applies to any $k$-uniform linear hypergraph that either (1) contains no $4$-cycles,\footnote{For graphs, Allen, Dhawan, and Noel establish this extension~\cite[Theorem~1.3]{high_girth}.} or (2) admits a degeneracy ordering that avoids certain configurations of $3$- and $4$-cycles. To keep the exposition of the proof ideas simple, we do not present the proof in this generality. Combining the extension in~(2) with a random sampling technique in the spirit of Duke, Lefmann, and R\"odl~\cite{duke1995uncrowded} to remove the forbidden configurations yields a previously unknown upper bound on the fractional chromatic number of all $d$-degenerate $k$-uniform linear hypergraphs with $k\ge3$ (see Corollary~\ref{corollary: degenerate linear}). We discuss this further in Section~\ref{section: concluding remarks}.

For the lower bound, we use a hypergraph variant of the uniform attachment model and harmonic vertex weights to give a lower bound on the fractional chromatic number via linear programming duality. We then remove all short cycles by deleting a small set of vertices with negligible total weight, preserving the asymptotically sharp lower bound. Unlike in graphs, the core vertices of a cycle do not determine its hyperedges, so we must sum over these additional choices. We bound their contributions and reduce the expected cycle count to a matrix trace. We then bound this trace using a matrix-norm estimate obtained via Schur's test, giving a simpler alternative to the Catalan-number argument in~\cite[proof of Lemma~3.3]{high_girth}. We give an overview in Section~\ref{subsection: lower bound overview}.

We remark that the hypergraphs constructed in Theorem~\ref{thm: degenerate}\ref{item: degenerate lb} possess independent sets strictly larger than $n/\chi_f(\cH)$. In particular, the lower bound in Theorem~\ref{thm: degenerate}\ref{item: degenerate lb} does not follow directly from a corresponding upper bound on $\alpha(\cH)$, and the degeneracy constraint plays an essential role in exhibiting this gap. We expand on this distinction in Section~\ref{section: proof overview}.

\subsubsection*{Fractional colorings with local demands.}
Motivated by applications to degree-sequence bounds on the independence number, Kelly and Postle~\cite{kelly2024fractional} introduced the theory of \emph{fractional coloring with local demands}, in which one seeks a probability distribution on independent sets such that each vertex $v$ is included with probability at least its prescribed demand $\phi(v)$. These demands may depend on local properties such as vertex degrees. More formally, we recall the following definition from~\cite[Definition~1.7]{dhawan2026fractional}.

\begin{definition}[$\phi$-colorings]
    Given a hypergraph $\cH$, a \emph{demand function} for $\cH$ is a function $\phi \colon V(\cH) \to [0, 1]$. A \emph{$\phi$-coloring} of $\cH$ is a probability distribution $\mathcal{D}$ over the independent sets of $\cH$ such that $\Pr_{I \sim \mathcal{D}}(v\in I) \ge \phi(v)$ for every vertex $v \in V(\cH)$.
\end{definition}

Our proof of Theorem~\ref{thm: degenerate}\ref{item: degenerate ub} naturally extends to fractional colorings with local demands; see Section~\ref{section: proof overview}.

\begin{theorem}\label{thm: local demand}
Let $k \ge 2$ be an integer. There exists a nonincreasing convex function $f_k\colon[0,\infty)\to[0,1]$ satisfying
\[
f_k(x)=(1-o_x(1))\left(\frac{1}{k-1}\,\frac{\log x}{x}\right)^{\frac{1}{k-1}}
\]
such that the following holds for every $n\in\N$. Let $\cH$ be an $n$-vertex $k$-uniform uncrowded hypergraph. Suppose there is a demand function $\phi$ for $\cH$ such that, for every $v\in V(\cH)$,
\[
\phi(v)\le f_k\left(\deg(v)\right).
\]
Then $\cH$ admits a $\phi$-coloring. Moreover, there is a $\mathrm{poly}(n)$-time algorithm that samples an independent set $I \subseteq V(\cH)$ such that, for each $v \in V(\cH)$, $\pr(v\in I) \ge \phi(v)$.
\end{theorem}

Theorem~\ref{thm: local demand} resolves a conjecture of Yu and Zhang~\cite[Conjecture~4.3]{yu2026independent}.

\subsubsection*{Degree-sequence bounds.}
Setting $\phi(v)=f_k(\deg(v))$, the maximum demand allowed by Theorem~\ref{thm: local demand}, and taking expectations gives $\alpha(\cH)\ge\sum_{v\in V(\cH)}f_k(\deg(v))$. The asymptotic behavior of $f_k$ then yields the following immediate corollary:

\begin{corollary}\label{corollary: local iset}
    Let $k \ge 2$ be an integer, and let $\eps>0$ be arbitrary. There exists $\delta_0 \in \N$ such that the following holds for every $\delta \ge \delta_0$ and $n \in \N$. Let $\cH$ be an $n$-vertex $k$-uniform uncrowded hypergraph of minimum degree $\delta$. Then
    \[
    \alpha(\cH) \ge (1-\eps)\sum_{v \in V(\cH)} {\left(\frac{1}{k-1} \, \frac{\log \deg(v)}{\deg(v)}\right)^{\frac{1}{k-1}}}.
    \]
\end{corollary}

Degree-sequence lower bounds on the independence number date back to Caro and Wei~\cite{caro1979new}, who established a degree-sequence analog of the classic greedy bound for graphs~\eqref{eq:greedy}. The first asymptotic improvement was proved in a seminal 1991 paper by Shearer~\cite{shearer1991note}, who established a degree-sequence bound valid for every triangle-free graph $G$. In its asymptotic form, this bound states that, as the minimum degree $\delta(G)$ tends to infinity,
\begin{equation}\label{eq: shearer}
    \alpha(G) \ge (1-o(1))\sum_{v\in V(G)}\frac{\log \deg(v)}{\deg(v)}.
\end{equation}
Shearer's proof introduced a technique now known as Shearer-style induction. Li, Rousseau, and Zang~\cite{li2002lower} extended this approach to locally sparse graphs—where the neighborhood of each vertex contains relatively few edges—and subsequent work established related results for broader graph classes under stronger minimum-degree assumptions (e.g., $\delta(G) \ge \mathsf{polylog}(\Delta(G))$~\cite{DKPS, bonamy2022bounding}). However, extending this technique from graphs to hypergraphs requires fundamentally new ideas. For $k \ge 3$, Corollary~\ref{corollary: local iset} provides the first result that matches the leading constant in~\eqref{eq: shearer} for any regime of $\delta(\cH)$, albeit under the assumption of uncrowdedness rather than girth at least $4$.

Finally, since $f_k$ is convex on $[0,\infty)$, Theorem~\ref{thm: local demand} and Jensen's inequality imply that
\begin{align*}
\alpha(\cH)&\ge\sum_{v\in V(\cH)}f_k\left(\deg(v)\right)\ge n f_k\left(\bar{d}\right) =(1-o_{\bar{d}}(1))\,n \,\left(\frac{1}{k-1}\,\frac{\log \bar{d}}{\bar{d}}\right)^{\frac{1}{k-1}},
\end{align*}
for every $n$-vertex $k$-uniform uncrowded hypergraph $\cH$ of average degree $\bar{d} \to \infty$.
As a consequence, this recovers the average-degree bound on the independence number of uncrowded hypergraphs obtained by Dhawan, Methuku, and Vo~\cite[Corollary~1.7]{DMV26ind} and, independently, Yu and Zhang~\cite[Theorem~1.1]{yu2026independent} via a substantially simplified argument.

\noindent \textbf{Note added.}
At the time of submission, we learned that Yu and Zhang~\cite{yu2026trianglefree} had independently proved the upper bound in Theorem~\ref{thm: degenerate}\ref{item: degenerate ub} and the local demand result in Theorem~\ref{thm: local demand}. Their proof uses a different approach, based on independent set sampling in the style of Martinsson and Steiner~\cite{martinsson2025random}, whereas our work builds on (and simplifies) the methods of Dhawan~\cite{dhawan2026fractional} and Allen, Dhawan, and Noel~\cite{high_girth}. Their results also apply to linear hypergraphs with no $4$-cycles, and our argument extends to this setting with minor modifications. For simplicity of exposition, we state our upper bounds for uncrowded hypergraphs and discuss the broader setting after the statement of Theorem~\ref{thm: degenerate}.

Beyond these upper bounds, the present paper contains several results not obtained in~\cite{yu2026trianglefree}. We prove the matching lower bound for hypergraphs of arbitrarily large fixed girth, and we derive consequences for $d$-degenerate linear hypergraphs. For the latter, we use the fact that our analysis only requires excluding certain configurations of $3$- and $4$-cycles relative to a degeneracy ordering. Combining this weaker requirement with random sampling and deletion yields Corollary~\ref{corollary: degenerate linear}. These consequences do not follow directly from the bound for linear hypergraphs with no $4$-cycles.

\subsection{Notation and terminology}\label{subsection: notation}

Throughout the rest of the paper, we use the following basic notation. For $n \in \N$, we let $[n] \defeq \set{1, \ldots, n}$.
For a hypergraph $\cH$, its vertex and edge sets are denoted by $V(\cH)$ and $E(\cH)$, respectively.
For a vertex $v \in V(\cH)$, $\deg_{\cH}(v) \coloneqq |\set{e \in E(\cH)\,:\, v \in e}|$ denotes the degree of $v$; we drop the subscript $\cH$ when the context is clear.
We say $u \neq v$ is a neighbor of $v$ if $\set{u , v} \subseteq e$ for some edge $e\in E(\cH)$.
For a subset $U \subseteq V(\cH)$, the subhypergraph induced by $U$ is denoted by $\cH[U]$.

We say an $n$-vertex hypergraph $\cH$ is \emph{$d$-degenerate} for $d\in \N$ if there exists an ordering $(v_1, \ldots, v_n)$ of $V(\cH)$ such that for each $i\in [n]$, we have $\deg_{\cH_i}(v_i) \leq d$, where $\cH_i \coloneqq \cH[\{v_1, \ldots, v_i\}]$.
Fix an ordering $(v_1,\ldots,v_n)$ of $V(\cH)$.  
For an edge $e \in E(\cH)$ and integers $0 \le i < j \le n+1$, define $e_{(i,j)}\coloneqq e\cap\{v_{i+1},\ldots,v_{j-1}\}$.
For each $i \in [n]$, we use the shorthand $e_{<i}$, $e_{\le i}$, $e_{>i}$, and $e_{\ge i}$ for the corresponding special cases: $e_{<i}\coloneqq e_{(0,i)}$, $e_{\le i}\coloneqq e_{(0,i+1)}$, $e_{>i}\coloneqq e_{(i,n+1)}$, and $e_{\ge i}\coloneqq e_{(i-1,n+1)}$.
If $\cH$ is linear, then any two vertices $v_i$ and $v_j$ in $\cH$ are contained in at most one edge. Thus, whenever such an edge exists, it is unique, and we denote it by $e^{ij}$.

For each $i\in[n]$, we define the \emph{right-neighborhood} of $v_i$, denoted by $N_R(v_i)$, to be the set of all vertices $v_j$ with $j > i$ such that there exists an edge $e \in E(\cH)$ with $e_{\ge i}=\{v_i,v_j\}$. 
In this case, $v_i$ and $v_j$ are, respectively, the \emph{penultimate} and \emph{final} vertices of $e$. We refer to the vertices in $N_R(v_i)$ as the \emph{right-neighbors} of $v_i$. Similarly, for each $i\in[n]$, we define the \emph{left-neighborhood} of $v_i$, denoted by $N_L(v_i)$, to be the set of all vertices $v_j$ for which $v_i$ is a right-neighbor of $v_j$, that is, $N_L(v_i) \coloneqq \{v_j : v_i\in N_R(v_j)\}$; we also 
refer to the vertices in $N_L(v_i)$ as the \emph{left-neighbors} of $v_i$. For $U\subseteq V(\cH)$, we write $N_R(U)\coloneqq\bigcup_{v\in U}N_R(v)$ and $N_L(U)\coloneqq\bigcup_{v\in U}N_L(v)$.

\subsection{Organization of the paper}\label{subsection: organization}

In Section~\ref{section: proof overview}, we give a sketch of the proof of Theorem~\ref{thm: degenerate}, outlining the main ideas and how our approach differs from earlier methods. In Section~\ref{section: fractional coloring}, we prove the upper bound in Theorem~\ref{thm: degenerate}\ref{item: degenerate ub} and the local-demand result, Theorem~\ref{thm: local demand}, by analyzing our independent set sampling procedure via a second-moment bound. Section~\ref{section: lower bound} proves the matching lower bound in Theorem~\ref{thm: degenerate}\ref{item: degenerate lb}. Finally, Section~\ref{section: concluding remarks} discusses extending the local-demand bound to linear hypergraphs and describes forthcoming work on triangle-free hypergraphs.

\section{Proof overview}\label{section: proof overview}

In this section, we provide an overview of the key ideas behind the proofs of our main results. Our goal is to provide intuition behind our method and highlight how it differs from previous approaches, without focusing on technical details.

\subsection{Upper bound}\label{subsection: upper bound overview}

The proof of Theorem~\ref{thm: degenerate}\ref{item: degenerate ub} builds on the random independent set procedure used in Dhawan's proof of Theorem~\ref{theo: hypergraph girth 4}~\cite{dhawan2026fractional} (inspired by Johansson's seminal argument for list coloring triangle-free graphs~\cite{Joh_triangle}) and the second-moment analysis from Allen, Dhawan, and Noel's proof of Theorem~\ref{theo: girth}\ref{item: high girth ub graph}~\cite{high_girth}.
We begin by reviewing these methods in the graph setting before outlining the key challenges and new ideas required for the hypergraph setting.

Let $G$ be a $d$-degenerate triangle-free graph. The procedure constructs an independent set $I \subseteq V(G)$ such that $\pr(v \in I)$ is large for every vertex $v$, as follows:
\begin{enumerate}
    \item Assign each vertex $v$ an initial weight $p_0(v) = \alpha$ for a suitably chosen parameter $\alpha \in [0, 1]$.
    \item Iterate through the vertices in a degeneracy ordering of $G$. When processing vertex $v_i$, include $v_i$ in $I$ with probability $p_{i-1}(v_i)$.
    \item For each $j > i$, update the weight $p_i(v_j)$ based on whether $v_i \in I$:
    \begin{itemize}
        \item If $v_i \in I$ and $v_j \in N_R(v_i)$, set $p_i(v_j) = 0$; this ensures $I$ remains an independent set.
        \item If $v_i \notin I$ and $v_j \in N_R(v_i)$, increase $p_i(v_j)$ to compensate and boost the probability of selecting $v_j$ later.
    \end{itemize}
\end{enumerate}
The parameter choices ensure the martingale property $\E[p_i(v_j) \mid p_{i-1}(\cdot)] = p_{i-1}(v_j)$ for $1\le i<j\le n$, which immediately implies that $\pr(v_j \in I) = \E[p_{j-1}(v_j)] = \alpha$.

To ensure that all weights remain valid probabilities (i.e., $p_{i-1}(v_i) \le 1$), weights are artificially capped by setting $p_i(v) = \hat p$ whenever $p_i(v) > \hat{p}$ for a well-chosen $\hat{p} \in [0, 1]$, classifying such vertices as \emph{bad}. Bad vertices are never added to $I$. Although capping breaks the martingale property, previous approaches restored it using an \emph{equalizing coin flip} (see~\cite[Ch.~13]{MolloyReed} for thresholding in the context of list-coloring triangle-free graphs).

Accounting for bad vertices yields:
\begin{align}\label{eq: lb on inclusion prob}
    \pr(v_j \in I) = \alpha - \hat{p} \pr(v_j \text{ is bad}).
\end{align}
To show that $\pr(v_j \in I)$ is large, it suffices to bound $\pr(v_j \text{ is bad})$. We do so by studying the random variable $X \coloneqq p_{j-1}(v_j)$. Although $\mathrm{Var}(X)$ is not generally small for triangle-free graphs, it was shown in~\cite{dhawan2026fractional} that the expected entropy $\E[-X \log X]$ is sufficiently large to complete the proof of Theorem~\ref{theo: martinsson steiner}. In~\cite{high_girth}, it was shown that $C_4$-freeness yields the second-moment bound $\E[X^2] \leq (1+o(1))\alpha^2 e^{(1+o(1))d\alpha}$, which gives
\[
\hat{p} \, \pr(v_j \text{ is bad})
= \frac{1}{\hat{p}}\left(\hat{p}^2 \, \pr(v_j \text{ is bad})\right)
\leq \frac{\E[X^2]}{\hat{p}}
\leq \frac{(1+o(1))\alpha^2 e^{(1+o(1))d\alpha}}{\hat{p}}.
\]
Combining this with \eqref{eq: lb on inclusion prob} gives:
\[
\pr(v_j \in I) \geq \alpha\left(1 - \frac{(1+o(1))\alpha e^{(1+o(1))d\alpha}}{\hat{p}}\right).
\]
Setting $\alpha = (1-o(1))\frac{\log d}{d}$ and $\hat{p} = \alpha^{o(1)}$ yields the optimal bound for graphs of girth at least $5$ (Theorem~\ref{theo: girth}).

To prove Theorem~\ref{thm: degenerate}\ref{item: degenerate ub}, we aim to control $\E[X^2]$ under the uncrowdedness assumption, generalizing the approach in~\cite{high_girth}.
In many nibble procedures, including those of Frieze and Mubayi~\cite{frieze2013coloring} and Cooper and Mubayi~\cite{cooper2015list}, uncolored vertices first activate tentative colors independently.
These choices determine which colors remain available: for an uncolored edge $e$, a color $c$ is removed at $v\in e$ if every vertex of $e\setminus\{v\}$ activates $c$, even if these choices never become permanent assignments.
The procedure in~\cite{dhawan2026fractional} uses a similar idea without making explicit tentative selections: if $v_j$ is the last vertex of an edge $e$, the rule for updating $p_i(v_j)$ involves products of the current weights of the other unprocessed vertices of $e$.
Thus the weight at $v_j$ can change before the penultimate vertex of $e$ is processed.
However, a direct extension of this approach does not yield the asymptotically optimal bound in Theorem~\ref{thm: degenerate}\ref{item: degenerate ub}.

Specifically, for a $k$-uniform $d$-degenerate uncrowded hypergraph $H$, directly applying the hypergraph procedure from~\cite{dhawan2026fractional} yields $\E[X^2] \leq (1+o(1))\alpha^2 e^{(1+o(1))(k-1)d\alpha^{k-1}}$. A calculation similar to the one above implies
\[
\chi_f(H) \leq (1+o(1))\left((k-1)^2 \frac{d}{\log d}\right)^{\frac{1}{k-1}},
\]
by setting $\alpha = (1-o(1))\left(\frac{1}{(k-1)^2}\frac{\log d}{d}\right)^{\frac{1}{k-1}}$ and $\hat{p} = \alpha^{o(1)}$. This bound exceeds the optimal result in Theorem~\ref{thm: degenerate}\ref{item: degenerate ub} by a factor of $(k-1)^{\frac{1}{k-1}}$.

We eliminate this extra factor by modifying the update rule from~\cite{dhawan2026fractional}. Let $e \in E(H)$ contain $v_j$, where $v_j$ is the last vertex of $e$ under the degeneracy ordering. In ~\cite{dhawan2026fractional}, processing any earlier vertex of $e$ can change the weight at $v_j$. Our algorithm instead updates $p_i(v_j)$ \textit{only} when processing the penultimate vertex of $e$ (see Algorithm~\ref{algorithm: iset uncrowded}).
At this stage, the selection outcomes of all earlier vertices in $e$ are known; thus, each update reflects a genuine risk that $e$ will be fully included in $I$. The degeneracy ordering naturally enables this delayed update mechanism and its subsequent analysis.

This modification improves the second-moment bound to \[\E[X^2] \leq (1+o(1))\alpha^2 e^{(1+o(1))d\alpha^{k-1}},\] 
from which a calculation similar to the one above yields
the proof of Theorem~\ref{thm: degenerate}\ref{item: degenerate ub} upon setting $\alpha = (1-o(1))\left(\frac{1}{k-1}\frac{\log d}{d}\right)^{\frac{1}{k-1}}$ and $\hat{p} = \alpha^{o(1)}$. 

Notably, our procedure also substantially simplifies the algorithm in~\cite[Algorithm~4.1]{dhawan2026fractional}.
First, in~\cite{dhawan2026fractional}, positive weights may decrease, and the analysis requires them not to fall below a prescribed lower threshold.
Vertices whose weights reach this threshold are declared bad and their weights are frozen; additional equalizing coin flips preserve the expected weights.
In our procedure, a weight can only increase, remain unchanged, or become zero, so every nonzero weight is at least its initial value.
This removes the need for the lower threshold, the associated bad vertices, and the corresponding equalizing coin flips.
Second, while we still employ an upper threshold to define bad vertices, we freeze their weights at their current values rather than truncating them to the threshold.
This avoids the additional coin flips needed to compensate for the change in expected weights caused by truncation.
These simplifications substantially streamline the subsequent second-moment analysis while also yielding a cleaner first-moment analysis.

Finally, we prove our result on fractional coloring with local demands, Theorem~\ref{thm: local demand}, using the same framework. We process vertices in nondecreasing order of demand rather than in a degeneracy ordering. We also allow the threshold for declaring a vertex bad to vary between vertices, so that the relative error at each vertex depends only on its degree.

\subsection{Lower bound}\label{subsection: lower bound overview}
To prove the lower bound in Theorem~\ref{thm: degenerate}\ref{item: degenerate lb}, we adapt the lower-bound strategy of~\cite{high_girth} to hypergraphs.
The basic strategy in~\cite{high_girth} is to use harmonic vertex weights in a uniform attachment model to construct a large fractional clique (the dual problem to fractional coloring; see~\eqref{eq:fracclique}), and then remove short cycles by deleting vertices of negligible total weight. 

Carrying out this strategy in the hypergraph setting introduces two main difficulties. First, bounding the probability that a set is independent requires counting $(k-1)$-subsets of earlier vertices.
We combine Jensen's inequality with the harmonic-weight estimate from~\cite{high_girth} to control these counts without losing the sharp leading constant.
Second, the core vertices of a cycle---one distinct vertex chosen from each intersection of consecutive hyperedges---do not determine its hyperedges.
We sum over the possible hyperedges according to their largest vertices, reducing the expected cycle count to a matrix trace.
We bound this trace using Schur's test, giving a simpler alternative to the Catalan-number argument in~\cite{high_girth}.

We now explain the argument in more detail, highlighting the challenges that arise and how we overcome them. We use the following variant of the \emph{uniform attachment model} (see, e.g.,~\cite{AcanPittel20, FriezePerezGimenezPralatReiniger19, MalyshkinZhukovskii22} for the graph setting and \cite{wang2010evolving, barthelemy2022class} for hypergraphs).
We begin with the vertex set $\{0,\dots,n\}$.
Independently for each $i\in[n]$, we choose $d$ distinct $(k-1)$-subsets of $\{0,\dots,i-1\}$ uniformly at random, taking all such subsets if fewer than $d$ exist.
For each chosen subset $e'$, we add the hyperedge $e \coloneqq e' \cup \{i\}$.
Finally, we delete vertex $0$ and all hyperedges containing it.
The resulting $k$-uniform hypergraph $\cH \sim U_{n,d}^{(k)}$ has vertex set $[n]=\{1,\dots,n\}$ (see Definition~\ref{def:model} for a formal description). 

In contrast to ordinary hypergraphs, the standard lower bound $\chi_f(\cH) \ge \frac{n}{\alpha(\cH)}$ is insufficient here. Although this random hypergraph contains independent sets well above size $n\left(\frac{\log d}{(k-1)d}\right)^{\frac{1}{k-1}}$, these sets are heavily biased—they are concentrated among the vertices with larger labels, $\{n/2, \dots, n\}$, while containing very few vertices in $\{1, \dots, n/2\}$. Consequently, they cannot be fractionally packed to yield a tight fractional coloring.

To bypass this barrier, we exploit linear programming duality. The dual program to the fractional chromatic number is the \emph{fractional clique number} $\omega_f(\cH)$, defined by the linear program:
\begin{maxi}
    {f \in \R^{|V(\cH)|}}{\sum_{v\in V(\cH)}f_v}{\label{eq:fracclique}}{}
    \addConstraint{\sum_{v\in I}f_v}{\le 1}{\qquad \text{for all } I \in \mathcal{I}(\cH)}
    \addConstraint{f_v}{\ge 0}{\qquad \text{for all } v \in V(\cH).}
\end{maxi}

By strong LP duality, $\omega_f(\cH) = \chi_f(\cH)$ for all hypergraphs $\cH$. Thus, to establish our lower bound on $\chi_f(\cH)$, it suffices by weak duality ($\omega_f(\cH) \le \chi_f(\cH)$) to construct an explicit, deterministic weight assignment $f \colon V(\cH) \to \mathbb{R}_{\ge 0}$ and show that, with high probability, $f$ forms a valid fractional clique in $\cH$ with sufficiently large value.

More formally, to each vertex $i \in [n]$, we assign the weight
\[
f(i) = \frac{\ell_n(i)}{(1+\eps)\,n}\left((k-1)\,\frac{d}{\log d}\right)^{\frac{1}{k-1}}, \qquad \text{where} \quad \ell_n(i) = \sum_{j = i}^n\frac{1}{j} \approx \log\left(\frac{n}{i}\right).
\]
Note that $f(i)$ decreases monotonically in $i$, meaning larger (later) vertices receive smaller weight.
Our proof now proceeds in two key steps:

\begin{enumerate}[wide, label=\textup{Step~\arabic*:}]
    \item \textit{Bounding the probability of a heavy independent set.}
    We show that the probability that there exists an independent set $I$ with $\sum_{i \in I}f(i) > 1$ is $o_n(1)$ (see Lemma~\ref{lem:heavy-independent}); i.e., $f$ is a fractional clique in $\cH$ with high probability.
    Specifically, for any fixed nonempty vertex subset $S \subseteq [n]$ satisfying $\frac{\ell_n(S)^2+|S|^2}{2n}\ge (k-1)H_n$, we establish in Proposition~\ref{prop:fixed-set} that
    \[
    \pr\left(S~\text{forms an independent set in}~\cH\right) \le \exp\left(-d|S|^{2-k} \left(\frac{\ell_n(S)^2+|S|^2}{2n}-(k-1)H_n\right)^{k-1}\right),
    \]
    where $\ell_n(S) = \sum_{i \in S}\ell_n(i)$ and $H_n = \sum_{i = 1}^n\frac{1}{i} \approx \log n$.
    Since $\sum_{i \in S}f(i) > 1$ if and only if $\ell_n(S)$ is large, this bound yields the desired concentration result.

    Proposition~\ref{prop:fixed-set} represents the first technical departure from the proof of Theorem~\ref{theo: girth}\ref{item: high girth lb graph} by Allen, Dhawan, and Noel. In the graph model, whether $S$ remains independent is determined sequentially by its earlier vertices. For $k$-uniform hypergraphs with $k \ge 3$, independence is governed simultaneously by all $(k-1)$-subsets in the initial segments of $S$. We estimate the binomial coefficients counting these subsets and apply Jensen's inequality, reducing the calculation to a lower bound on $\sum_{j\in S}|S\cap[j]|/j$.
    Proposition~\ref{prop:harmonic-energy} bounds this sum below by $(\ell_n(S)^2+|S|^2)/(2n)$, giving the displayed probability bound without losing a constant factor.

    For $g = 2$, the proof is complete at this stage, as every hypergraph has girth at least $2$ and
    \[
    \sum_{i = 1}^nf(i) = \frac{1}{(1+\eps)}\left((k-1)\,\frac{d}{\log d}\right)^{\frac{1}{k-1}} \ge (1 - \eps)\left((k-1)\,\frac{d}{\log d}\right)^{\frac{1}{k-1}}.
    \]

    \item \textit{Removing short cycles while essentially preserving total weight.}
    Restricting $f$ to any subhypergraph $\cH' \coloneqq \cH[U]$ for $U \subseteq V(\cH)$ yields a valid fractional clique in $\cH'$, as any independent set in $\cH'$ remains independent in $\cH$.
    Let $R \subseteq [n]$ denote the set of vertices that serve as the rightmost (i.e., largest) vertex of at least one short cycle in $\cH$, and define $U \coloneqq [n] \setminus R$.
    It remains to show that $\sum_{i \in R} f(i) = o(1)$, which ensures that the total weight removed by deleting $R$ is negligible.
    Since $\sum_{i \in R} f(i) = 
    O(\ell_n(R)/n)$, it suffices to establish that $\ell_n(R) = o(n)$.
    Indeed, using Markov's inequality, we show that this property holds with constant probability by proving that $|R|$ itself is small.

    A key step in the argument bounds the expected number of short cycles: specifically, $\E[C_q] = 
    O(\log n)
    $, where $C_q$ denotes the number of $q$-cycles in $\cH$.
    For a fixed sequence of $q$ vertices $v_1, \ldots, v_q$, the probability that they form a $q$-cycle is roughly proportional to the product of the probabilities that $v_i$ and $v_{i+1}$ (indices taken modulo $q$) share an edge.
    In the graph setting ($k = 2$), this probability is at most $d/\max\{v_i, v_{i+1}\}$.
    For $k$-uniform hypergraphs with $k \ge 3$, the computation becomes substantially more delicate because $v_i$ and $v_{i+1}$ may belong to an edge $e$ whose maximum vertex satisfies $\max e > \max\{v_i, v_{i+1}\}$.
    Nevertheless, a careful combinatorial counting argument yields the upper bound $(k-1)^2 d / \max\{v_i, v_{i+1}\}$ for this adjacency probability, resulting in
    \[
    \E[C_q] \le (k-1)^{2q} d^q \sum_{v_1, \ldots, v_q} \prod_{i = 1}^q \frac{1}{\max\{v_i, v_{i+1}\}}, \qquad \text{where} \quad v_{q+1} \coloneqq v_1.
    \]
    Defining $A$ to be the $n \times n$ matrix with entries $A_{ij} = 1/\max\{i, j\}$, the sum above is precisely the trace $\Tr(A^q)$.
    In \cite{high_girth}, $\Tr(A^q) =
    O(\log n)$ was established via an intricate combinatorial analysis relying on Catalan numbers.
    Here, we present an alternative and considerably simpler proof of this bound via matrix-norm analysis.
\end{enumerate}

\section{Fractional colorings of $d$-degenerate uncrowded hypergraphs}\label{section: fractional coloring}

In this section, we will prove Theorems~\ref{thm: degenerate}\ref{item: degenerate ub} and~\ref{thm: local demand}.
We begin by stating a key result, which shows that such a hypergraph admits a $\phi$-coloring with respect to an appropriate demand function $\phi$.

\begin{proposition}\label{prop: degenerate + local demand}
Let $k \ge 2$ be an integer, and let $\eps>0$ be arbitrary. There exists $\delta_0 \in \N$ such that the following holds for every $\delta \ge \delta_0$ and $n \in \N$.
Let $\cH$ be an $n$-vertex $k$-uniform uncrowded hypergraph with a vertex ordering $(v_1,\dots, v_n)$.
For each $i \in [n]$, let $\cH_i \coloneqq \cH[\{v_1, \ldots, v_i\}]$.
Suppose that there is a demand function $\phi$ for $\cH$ such that:
\begin{enumerate}[(i)]
\item $\phi(v_i)$ is nondecreasing in $i$; and
\item for every $i \in [n]$,
\begin{equation}\label{eq: demand bound}
\phi(v_i) \le \frac{1}{1+\eps} {\left(\frac{1}{k-1} \, \frac{\log d^*(v_i)}{d^*(v_i)}\right)^{\frac{1}{k-1}}},
\end{equation}
where $d^*(v_i) \coloneqq \max\left\{\deg_{\cH_i}(v_i),\,\delta\right\}$.
\end{enumerate}
Then $\cH$ admits a $\phi$-coloring.
Moreover, there is a $\poly(n)$-time algorithm that samples an independent set $I \subseteq V(\cH)$ such that, for each $v \in V(\cH)$, $\pr(v\in I) \ge \phi(v)$.
\end{proposition}

Before proving Proposition~\ref{prop: degenerate + local demand}, we first show how it implies Theorem~\ref{thm: degenerate}\ref{item: degenerate ub} and a weaker version of Theorem~\ref{thm: local demand}.

\begin{proof}[Proof of Theorem~\ref{thm: degenerate}\ref{item: degenerate ub}] Consider a degeneracy ordering $(v_1,\dots,v_n)$ of $\cH$. Set $\delta=d$, and consider the demand function $\phi$ for $\cH$ given by
\[
\phi(v) = \frac{1}{1+\eps}{\left(\frac{1}{k-1} \,\frac{\log d}{d}\right)^{\frac{1}{k-1}}},
\]
for every $v\in V(\cH)$. The hypotheses of Proposition~\ref{prop: degenerate + local demand} are then satisfied, and Proposition~\ref{prop: degenerate + local demand} yields Theorem~\ref{thm: degenerate}\ref{item: degenerate ub}.
\end{proof}

\begin{corollary}\label{thm: local demand eps}
    Let $k \ge 2$ be an integer, and let $\eps>0$ be arbitrary. There exists $\delta_0 \in \N$ such that the following holds for every $\delta \geq \delta_0$ and $n \in \N$. Let $\cH$ be an $n$-vertex $k$-uniform uncrowded hypergraph of minimum degree $\delta$. Suppose there is a demand function $\phi$ for $\cH$ such that, for each $v \in V(\cH)$,
    \[
    \phi(v) \leq (1-\eps) {\left(\frac{1}{k-1} \, \frac{\log \deg(v)}{\deg(v)}\right)^{\frac{1}{k-1}}}.
    \]
    Then $\cH$ admits a $\phi$-coloring. Moreover, there is a $\mathrm{poly}(n)$-time algorithm that samples an independent set $I \subseteq V(\cH)$ satisfying $\Pr[v\in I] \geq \phi(v)$ for each $v \in V(\cH)$.
\end{corollary}

\begin{proof}
Consider an ordering $(v_1,\dots,v_n)$ of $V(\cH)$ such that $\phi(v_1)\le\cdots\le\phi(v_n)$, and set $\delta=\delta(\cH)$. Since $\deg_{\cH_i}(v_i) \le \deg(v_i)$, it follows that $d^*(v_i) = \max\{\deg_{\cH_i}(v_i),\,\delta\} \le \deg(v_i)$ for every $v_i \in V(\cH)$.
As $t \mapsto \left((\log t)/ t\right)^{\frac{1}{k-1}}$ is decreasing for $t$ sufficiently large, and $1-\eps < 1/(1+\eps)$, the demand condition in Corollary~\ref{thm: local demand eps} implies the corresponding condition in Proposition~\ref{prop: degenerate + local demand}.
Hence, Proposition~\ref{prop: degenerate + local demand} yields Corollary~\ref{thm: local demand eps}.
\end{proof}

The proof of Theorem~\ref{thm: local demand} uses essentially the same argument as the proof of Corollary~\ref{thm: local demand eps}. However, since the error in Theorem~\ref{thm: local demand}  tends to zero as the degree grows, it requires a stronger version of Proposition~\ref{prop: degenerate + local demand} with a degree-dependent error. For ease of exposition, we first prove the proposition as stated and then return to the stronger version and Theorem~\ref{thm: local demand} in Section~\ref{subsection: local demand prop}.

\subsection{Proof of Proposition~\ref{prop: degenerate + local demand}} 
Fix $k \ge 2$ and $\eps > 0$. Without loss of generality, we may assume that $\eps < 1$. Set $\eta \coloneqq \eps/10$. Choose $\delta_0$ sufficiently large in terms of $k$ and $\eps$, and fix $\delta \ge \delta_0$. 

Let $\cH$ be an $n$-vertex $k$-uniform uncrowded hypergraph with a vertex ordering $(v_1,\dots,v_n)$ and a demand function $\phi$ satisfying the hypotheses of Proposition~\ref{prop: degenerate + local demand}. Note that $d^*(v_i) \ge \delta \ge \delta_0$ for every $i \in [n]$; by choosing $\delta_0$ sufficiently large, we may ensure that all subsequent inequalities requiring $d^*(v_i)$ to be large are satisfied. We will use this without further comment.

We begin with an informal description of our independent set procedure (see Algorithm~\ref{algorithm: iset uncrowded} for the formal statement).
The algorithm takes as input an $n$-vertex $k$-uniform uncrowded hypergraph $\cH$, a demand function $\phi$ for $V(\cH)$, a vertex ordering $(v_1, \ldots, v_n)$ nondecreasing with respect to $\phi$, and a parameter $\eta > 0$, and returns an independent set $I \subseteq V(\cH)$. Each vertex $v_i$ is initially assigned the weight $p_0(v_i) = \alpha_{v_i} \coloneqq \frac{\phi(v_i)}{1 - 2\eta/k}$.
Iterating through the vertices according to the specified ordering, we include each vertex $v_i$ in $I$ with probability $p_{i-1}(v_i)$. 

Recall from Section~\ref{subsection: notation} that $v_j \in N_R(v_i)$ precisely when $v_i$ and $v_j$ are the penultimate and final vertices, respectively, of some edge under the given ordering. After processing $v_i$, we update the weights $p_i(v_j)$ for each $j > i$ based on whether $v_i \in I$:
\begin{itemize}
\item If $v_j \notin N_R(v_i)$ or $e_{< i}^{ij} \not\subseteq I$, we leave the weight unchanged: $p_i(v_j) = p_{i-1}(v_j)$.
\item If $v_j \in N_R(v_i)$ and $e_{< i}^{ij} \subseteq I$, we set
\begin{equation}\label{eq: update intro}
p_i(v_j) = \begin{cases}
0 & \text{if } v_i \in I, \\
\dfrac{p_{i-1}(v_j)}{1 - p_{i-1}(v_i)} & \text{if } v_i \notin I.
\end{cases}
\end{equation}
\end{itemize}
In particular, if $u \in I$ for every $u \in e^{ij} \setminus \{v_j\}$, then either $v_j$ is already bad or the update rule gives $p_{j-1}(v_j) = 0$, ensuring that $v_j$ is never added to $I$.

As noted in the overview (Section~\ref{subsection: upper bound overview}), a vertex $v_j$ is classified as \emph{bad} if its weight becomes too large—specifically, if $p_i(v_j) \ge \alpha_{v_j}^{\eta}$. Bad vertices are excluded from $I$, and once a vertex becomes bad, its weight is frozen and no longer updated.

\smallskip

Let us now provide a formal description of the algorithm. 

\vspace{5pt}
\begin{breakablealgorithm}
\caption{Random Independent Set Procedure for Uncrowded Hypergraphs}\label{algorithm: iset uncrowded}
\begin{flushleft}
\noindent\makebox[4.5em][l]{\textbf{Input}:}%
\parbox[t]{\dimexpr\linewidth-4.5em\relax}{%
An $n$-vertex $k$-uniform uncrowded hypergraph $\cH$, a demand function $\phi$ for $V(\cH)$, a vertex ordering $(v_1, \ldots, v_n)$ nondecreasing with respect to $\phi$, and a parameter $\eta > 0$.}

\smallskip

\noindent\makebox[4.5em][l]{\textbf{Output}:}%
\parbox[t]{\dimexpr\linewidth-4.5em\relax}{%
An independent set $I \subseteq V(\cH)$.
}
\end{flushleft}

\begin{enumerate}[itemsep=3pt, label=(\Alph{enumi}), ref=(\Alph{enumi})]
\item \textbf{Initialize:} Set $I = B_0 = \emptyset$.
For each $i \in [n]$, set $p_0(v_i) = \alpha_{v_i}$, where $\alpha_{v_i} \coloneqq \frac{\phi(v_i)}{1 - 2\eta/k}$.

\item \textbf{Iterate:} For $i = 1, \ldots, n$, do the following:
\begin{enumerate}[itemsep=3pt, label=(\Roman{enumii}), ref=(\Alph{enumi})(\Roman{enumii})]
\item\label{step: bad skip} If $v_i\in B_{i-1}$, set $p_i(v_j)=p_{i-1}(v_j)$ for all $j$.
\item If $v_i\not\in B_{i-1}$, do the following:
\begin{enumerate}[itemsep=3pt, label=(\arabic{enumiii}), ref=(\Alph{enumi})(\Roman{enumii})(\arabic{enumiii})]
\item\label{step: no update} For each $j$ such that $v_j \notin N_R(v_i) \setminus B_{i-1}$, set $p_i(v_j) = p_{i-1}(v_j)$.
\item\label{step: coin flip} Let $a_{i} \sim \mathrm{BER}(p_{i-1}(v_i))$. If $a_{i} = 1$, then add $v_i$ to $I$.
\item For each $v_j \in N_R(v_i)\setminus B_{i-1}$, do the following:
\begin{enumerate}[itemsep=3pt, label=(\roman{enumiv}), ref=(\Alph{enumi})(\Roman{enumii})(\arabic{enumiii})(\roman{enumiv})]
\item\label{step: prefix fail} If $e^{ij}_{< i} \not\subseteq I$, set $p_i(v_j) = p_{i-1}(v_j)$.
\item\label{step: set zero} If $e^{ij}_{< i}\subseteq I$ and $a_{i} = 1$, set $p_i(v_j) = 0$.
\item\label{step: weight boost} If $e^{ij}_{< i}\subseteq I$ and $a_{i} = 0$, set $p_i(v_j) = \frac{p_{i-1}(v_j)}{1 - p_{i-1}(v_i)}$.
\end{enumerate}
\end{enumerate}
\item\label{step: bad update} Set $B_i=\{v_j\,:\,p_i(v_j)\ge \alpha_{v_j}^{\eta}\}$.
\end{enumerate}
\end{enumerate}
\end{breakablealgorithm}
\vspace{5pt}

Let $\cF_i$ denote the history of the procedure through the end of step $i$. Since the sequence $(\alpha_{v_i})_{i=1}^n$ is nondecreasing, the update rule gives, for $0 \le i < j \le n$,
\begin{equation}\label{eq:weight-cap}
0\le p_i(v_j)\le \frac{\alpha_{v_j}^{\eta}}{1-\alpha_{v_j}^{\eta}}.
\end{equation}
Indeed, before $v_j$ is included in $B_i$ its weight is smaller than $\alpha_{v_j}^{\eta}$, and a single update can multiply it by at most $1/(1-\alpha_{v_j}^{\eta})$; after it is included in $B_i$, its weight is no longer updated.\footnote{Allen, Dhawan, and Noel~\cite{high_girth} relied on equalizing coin flips to enforce $p_i(v_j) \leq \alpha_{v_j}^{\eta}$. Dispensing with these coin flips simplifies our analysis.}

\begin{lemma}\label{lemma: potential}
For each $j \in [n]$ and each $i \in \{1, \ldots, j-1\}$, we have $\E[p_i(v_j) \,|\, \cF_{i-1}] = p_{i-1}(v_j)$.
\end{lemma}

\begin{proof}
Fix $j \in [n]$ and $i \in [j-1]$.
If $v_i \in B_{i-1}$, then $p_i(v_j) = p_{i-1}(v_j)$ by Step~\ref{step: bad skip}.
Suppose $v_i \notin B_{i-1}$.
If $v_i \notin N_L(v_j)$, then $v_j \notin N_R(v_i)$, so $p_i(v_j) = p_{i-1}(v_j)$ by Step~\ref{step: no update}.
Suppose $v_i \in N_L(v_j)$, and let $e = e^{ij}$ be the unique edge containing both $v_i$ and $v_j$.
If $v_j \in B_{i-1}$, then $v_j \notin N_R(v_i) \setminus B_{i-1}$, so $p_i(v_j) = p_{i-1}(v_j)$ by Step~\ref{step: no update}.
Suppose $v_j \notin B_{i-1}$.
If $e_{< i} \not\subseteq I$, then $p_i(v_j) = p_{i-1}(v_j)$ by Step~\ref{step: prefix fail}.
Finally, suppose $e_{< i} \subseteq I$. Then by Steps~\ref{step: set zero} and~\ref{step: weight boost},
\begin{align*}
    \E[p_i(v_j) \,|\, \cF_{i-1}]
    &= p_{i-1}(v_i) \cdot 0 + (1 - p_{i-1}(v_i)) \cdot \frac{p_{i-1}(v_j)}{1 - p_{i-1}(v_i)} = p_{i-1}(v_j),
\end{align*}
completing the proof.
\end{proof}

The key property of the algorithm is given by the following lemma, which immediately yields Proposition~\ref{prop: degenerate + local demand}.

\begin{lemma}\label{lemma: prob} For each $j \in [n]$, we have
$\pr(v_j \in I) \ge (1 - 2\eta/k) \alpha_{v_j} = \phi(v_j)$.
\end{lemma}

\begin{proof}
\stepcounter{ForClaims}
\renewcommand{\theForClaims}{\ref*{lemma: prob}}
Fix $j \in [n]$. If $\phi(v_j) = 0$, the conclusion is immediate. If $j = 1$, the algorithm includes $v_1$ with probability $\alpha_{v_1} \geq \phi(v_1)$. We may therefore assume that $\phi(v_j) > 0$ and $j \geq 2$.
Given that $d^*(v_j) \ge \delta$ is sufficiently large, by the demand bound~\eqref{eq: demand bound} and the definition of $\alpha_{v_j}$, we first note that
\begin{equation}\label{eq: alpha v_j bound}
\alpha_{v_j} = \frac{\phi(v_j)}{1-2\eta/k} \le \frac{1}{(1-2\eta/k)(1+\eps)} \left(\frac{1}{k-1} \, \frac{\log d^*(v_j)}{d^*(v_j)}\right)^{\frac{1}{k-1}} \ll d^*(v_j)^{-\frac{1-\eta}{k-1}}.
\end{equation}

Now, if $v_j \in B_{j-1}$, then $v_j$ is not included in $I$ by Step~\ref{step: bad skip}; otherwise (i.e., if $v_j \notin B_{j-1}$), $v_j$ is included in $I$ with probability $p_{j-1}(v_j)$ in the $j$-th iteration. Thus,
\begin{equation}\label{eq: v_j in I first expression}
\pr(v_j\in I) = \E\left[p_{j-1}(v_j)\,\I\left[v_j \notin B_{j-1}\right]\right] = \E\left[p_{j-1}(v_j)\right] -  \E\left[p_{j-1}(v_j)\,\I\left[v_j \in B_{j-1}\right]\right].
\end{equation}
By Lemma~\ref{lemma: potential} and the tower property of conditional expectation, we have
\[
\E[p_{j-1}(v_j)] 
= \E\left[\E[p_{j-1}(v_j)\,|\,\cF_{j-2}]\right] 
= \E[p_{j-2}(v_j)] 
= \cdots 
= \E[p_0(v_j)] 
= p_0(v_j) 
= \alpha_{v_j}.
\]
Combining this with~\eqref{eq:weight-cap} and~\eqref{eq: v_j in I first expression}, we obtain that
\begin{align}
\pr(v_j\in I) 
&= \alpha_{v_j} - \E\left[p_{j-1}(v_j)\,\I\left[v_j \in B_{j-1}\right]\right]\\
&\ge \alpha_{v_j} - \frac{\alpha_{v_j}^{\eta}}{1 - \alpha_{v_j}^{\eta}}\,\pr(v_j \in B_{j-1}) 
\ge \alpha_{v_j} - 2\alpha_{v_j}^{\eta}\pr(v_j \in B_{j-1}),\label{eq: v_j in I, v_j bad}
\end{align}
where in the last inequality we used $\alpha_{v_j}^{\eta} \le 1/2$ by~\eqref{eq: alpha v_j bound} and that $d^*(v_j)$ is sufficiently large. 

As outlined in the proof overview (Section~\ref{subsection: upper bound overview}), we control $\pr(v_j \in B_{j-1})$ by bounding $\E[p_{j-1}(v_j)^2]$ from above.
Indeed, since $v_j \in B_{j-1}$ implies $p_{j-1}(v_j) \ge \alpha_{v_j}^{\eta}$, we have $\E[p_{j-1}(v_j)^2] \ge \alpha_{v_j}^{2\eta} \, \Pr(v_j \in B_{j-1})$, and so it follows from~\eqref{eq: v_j in I, v_j bad} that
\begin{equation}\label{eq: v_j in I middle bound}
\Pr(v_j\in I) \ge \left(1 - 2\alpha_{v_j}^{-1-\eta}\,\E[p_{j-1}(v_j)^2]\right)\alpha_{v_j}.
\end{equation}
It suffices therefore to bound $\E[p_{j-1}(v_j)^2]$. To this end, we define a few auxiliary random variables.
For each $v_\ell \in N_L(v_j)$ and $0 \leq i < \ell$, let
\begin{equation}\label{def: A_ell,i}
    A_{\ell,i} = 1 + (1+\eta)\, \I\left[e^{\ell j}_{\le i} \subseteq I\right]\prod_{u \in e^{\ell j}_{(i, j)}} p_{i}(u).
\end{equation}
With the above in hand, for each $0 \le i \le j-1$, define
\begin{equation}\label{def: S_i}
S_i \coloneqq p_{i}(v_j)^2 \prod_{\substack{i < \ell < j \\ v_\ell \in N_L(v_j)}} A_{\ell, i}.
\end{equation}
Note that only factors with $i < \ell$ occur in $S_i$.
The following claim, which is key to the proof, shows that $S_i$ is a supermartingale.
\begin{claim}\label{claim: S_i}
For $i = 1, \ldots, j-1$, we have $\E[S_i \,|\, \cF_{i-1}] \le S_{i-1}$.
\end{claim}

\begin{claimproof}
Fix $i\in[j-1]$. Recall that
\begin{equation}\label{eq: S_i recall}
S_i = p_{i}(v_j)^2\prod_{\substack{i < \ell < j \\ v_\ell \in N_L(v_j)}}A_{\ell,i}
\end{equation}
and
\begin{equation}\label{eq: S_{i-1} recall}
S_{i-1} = p_{i-1}(v_j)^2\prod_{\substack{i \le \ell < j \\ v_\ell \in N_L(v_j)}}A_{\ell,i-1}.
\end{equation}
Note that $S_i$ differs from $S_{i-1}$ in three respects: 
\begin{enumerate}
\item all weights $p_{i-1}(\cdot)$ are replaced by $p_i(\cdot)$;
\item for each $A_{\ell, i}$ indexed by $\ell > i$, the indicator changes from $\I\left[e^{\ell j}_{\leq i - 1} \subseteq I\right]$ to $\I\left[e^{\ell j}_{\le i} \subseteq I\right]$, and the index set of the product changes from $e^{\ell j}_{(i-1,j)}$ to $e^{\ell j}_{(i,j)}$ (both reflecting the transition past $v_i$); and
\item if $v_i \in N_L(v_j)$, the term $A_{i, i-1}$ is present in $S_{i-1}$ but absent from $S_i$.
\end{enumerate}

Let
\[
X \coloneqq \bigcup_{\substack{i < \ell < j \\ v_\ell \in N_L(v_j)}}e^{\ell j}_{(i, j)}
\]
denote the set of all vertices (other than $v_j$) whose weights appear in the definition of~$S_i$.
When we process $v_i$, $S_i$ can differ from $S_{i-1}$ only if (i) $v_i$ is a left-neighbor of some vertex in $X \cup \{v_j\}$, or (ii) $v_i$ is contained in some edge $e^{\ell j}$ with $i < \ell < j$.
More precisely, $v_i$ can play one of only four roles (see Fig.~\ref{fig: claim roles} for an illustration of the first three):
\begin{enumerate}[(a)]
\item\label{role: left-neighbor X} $v_i \in N_L(u)$ for some $u \in X$; letting $v_s \in N_L(v_j)$ be such that $e^{s j}$ is the unique edge containing $u$ and $v_j$, $p_i(u)$ may differ from $p_{i-1}(u)$ when considering $A_{s, i-1}$ vs $A_{s, i}$; or
\item\label{role: internal} $v_i \in e^{s j}$ for some $i < s < j$ with $v_{s} \in N_L(v_j)$, so that $\I\left[e^{s j}_{\le i} \subseteq I\right] = \I\left[e^{s j}_{\leq i-1} \subseteq I\right] \cdot \I\left[v_i \in I\right]$ when considering $A_{s, i-1}$ vs $A_{s, i}$; or
\item\label{role: left-neighbor vj} $v_i \in N_L(v_j)$, so that $p_i(v_j)$ may differ from $p_{i-1}(v_j)$, and the term $A_{i, i-1}$ is absent from $S_i$; or
\item\label{role: none} none of the above.
\end{enumerate}

\begin{figure}[htb!]
    \centering
    \vspace{-7.0pt}
    \begin{tikzpicture}[scale=1.1, every node/.style={scale=1.1}]
    
    \tikzset{
        vert/.style={circle, fill=black, draw=black, line width=0.15pt, inner sep=0pt, minimum size=2.5pt},
        rvert/.style={circle, fill=red!80!black, draw=red!80!black, line width=0.15pt, inner sep=0pt, minimum size=2.5pt},
        svert/.style={circle, fill=orange!80!yellow, draw=orange!70!yellow, line width=0.15pt, inner sep=0pt, minimum size=2.5pt},
    }


    \node[vert] (vj) at (5, 0) {};
    \node[font=\scriptsize, right=1pt] at (4.8, -0.45) {$v_j$};

    \draw (2.70, 0) ellipse (2.8 and 0.4);
    \node[font=\scriptsize] at (0, -0.5) {$e_{\ref{role: internal}} = e^{s j}$};

    \node[vert] at (0.40, 0) {};
    \node[font=\scriptsize] at (0.79, 0) {$\cdots$};
    \node[svert] at (1.17, 0) {};
    \node[font=\scriptsize, below=0pt] at (1.4, 0.15) {$v_i^{\ref{role: internal}}$};

    \fill[red!10, opacity=0.7] (3.08, 0) ellipse (1.3 and 0.25);
    \draw[red!50!black, dashed] (3.08, 0) ellipse (1.3 and 0.25);

    \node[rvert] at (1.93, 0) {};
    \node[rvert] at (2.70, 0) {};
    \node[font=\scriptsize, left=0pt] at (2.70, 0.095) {$u$};
    \node[rvert] at (3.47, 0) {};
    \node[rvert] at (4.23, 0) {};
    \node[font=\scriptsize, red!70!black, left=0pt] at (4.25, 0) {$v_s$};

    \draw[rotate around={-40:(3.238, 1.478)}] (3.238, 1.478) ellipse (2.8 and 0.3);

    \node[vert] at (1.476, 2.957) {};
    \node[font=\scriptsize, rotate=-40] at (2.063, 2.464) {$\cdots$};

    \fill[red!10, opacity=0.7, rotate around={-40:(3.532, 1.232)}]
         (3.532, 1.232) ellipse (1.3 and 0.25);
    \draw[red!50!black, dashed, rotate around={-40:(3.532, 1.232)}]
         (3.532, 1.232) ellipse (1.3 and 0.25);

    \node[rvert] at (2.651, 1.971) {};
    \node[rvert] at (3.238, 1.478) {};
    \node[rvert] at (3.825, 0.986) {};
    \node[rvert] at (4.413, 0.493) {};

    \draw[rotate around={40:(3.238, -1.478)}] (3.238, -1.478) ellipse (2.8 and 0.3);
    \node[font=\scriptsize] at (3.296, -2.260) {$e_{\ref{role: left-neighbor vj}}$};

    \node[vert] at (1.476, -2.957) {};
    \node[vert] at (2.063, -2.464) {};
    \node[vert] at (2.651, -1.971) {};
    \node[vert] at (3.238, -1.478) {};
    \node[font=\scriptsize, rotate=40] at (3.649, -1.134) {$\cdots$};
    \node[svert] at (4.413, -0.493) {};
    \node[font=\scriptsize, below left=0pt] at (4.503, -0.468) {$v_i^{\ref{role: left-neighbor vj}}$};

    \draw[rotate around={140:(0.938, 1.478)}] (0.938, 1.478) ellipse (2.8 and 0.3);
    \node[font=\scriptsize, left] at (-0.185, 2.006) {$e_{\ref{role: left-neighbor X}}$};
    
    \node[svert] at (2.113, 0.493) {};
    \node[font=\scriptsize, above left=0pt] at (2.230, 0.378) {$v_i^{\ref{role: left-neighbor X}}$};
    \node[vert] at (1.398, 1.093) {};
    \node[font=\scriptsize, rotate=-40] at (0.938, 1.478) {$\cdots$};
    \node[vert] at (0.351, 1.971) {};
    \node[vert] at (-0.237, 2.464) {};
    \node[vert] at (-0.824, 2.957) {};

    \node[font=\scriptsize, anchor=north west] at (-4.0, -1.5) {%
        \begin{tabular}{@{}l@{\hskip 4pt}l@{}}
        \tikz\node[circle,fill=orange!80!yellow,draw=orange!70!yellow,line width=0.15pt,inner sep=0pt,minimum size=2.5pt]{}; & possible positions for $v_i$\\[2pt]
        \tikz\node[circle,fill=red!80!black,draw=red!80!black,line width=0.15pt,inner sep=0pt,minimum size=2.5pt]{}; & vertices in $\{v_{i+1},\ldots,v_{j-1}\}$
        \end{tabular}};
    \end{tikzpicture}
    \caption{%
    An illustration of the possible roles of $v_i$ relative to $v_j$ and the edges $e^{\ell j}$ in the proof of Claim~\ref{claim: S_i}. Red vertices lie in $\{v_{i+1},\ldots,v_{j-1}\}$; orange vertices mark the three possible positions for $v_i$ (one per case). In Case~\ref{role: left-neighbor X}, $v_i$ is a left-neighbor of a vertex $u \in X$ (here $u \in e^{sj}_{(i,j)}$), with $e_{\ref{role: left-neighbor X}}$ being the unique edge containing $v_i$ and $u$. In Case~\ref{role: internal}, $v_i$ is a vertex of the edge $e^{sj}$, where $v_s \in N_L(v_j)$. In Case~\ref{role: left-neighbor vj}, $v_i$ is the penultimate vertex of $e_{\ref{role: left-neighbor vj}}$, so $v_i \in N_L(v_j)$.
    }
    \label{fig: claim roles}
\end{figure}
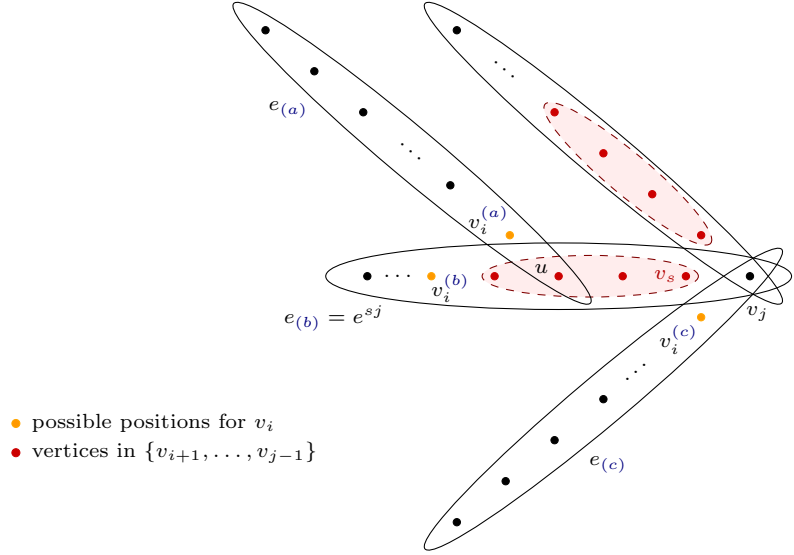

We now argue that these roles are mutually exclusive.
Since $\cH$ is uncrowded, $N_L(u) \cap N_L(w) = \varnothing$ for all distinct $u, w \in X \cup \{v_j\}$; otherwise, a common left-neighbor would form a cycle of length at most~$4$ together with the corresponding edges.
In particular, $v_i$ belongs to $N_L(w)$ for at most one $w \in X \cup \{v_j\}$, so \ref{role: left-neighbor X} and \ref{role: left-neighbor vj} are mutually exclusive.
For the same reason, \ref{role: left-neighbor X} cannot hold simultaneously for distinct $u, w \in X$.
Similarly, as $\cH$ is linear, $v_i$ lies in at most one edge $e^{s j}$ with $i < s < j$ and $v_{s} \in N_L(v_j)$.
Moreover, \ref{role: internal} cannot hold simultaneously with \ref{role: left-neighbor vj}, as two such edges would contain both $v_i$ and $v_j$, contradicting the linearity of $\cH$.
Finally, \ref{role: left-neighbor X} and \ref{role: internal} cannot hold together: if $v_l$ is the (unique) vertex in $N_R(v_i) \cap X$, then the edges $e^{ij}$, $e^{il}$, and $e^{lj}$ would form a $3$-cycle, again contradicting uncrowdedness.

We consider cases. First, suppose that $v_i \in B_{i-1}$. Then, the algorithm does not update any weights (Step~\ref{step: bad skip}), and $v_i$ is not added to $I$.
Since $p_i(v_j) = p_{i-1}(v_j)$ and $p_i(u) = p_{i-1}(u)$ for all $u \in X$, it remains to compare the terms $A_{\ell, i}$ vs $A_{\ell, i-1}$.
For each $\ell$ with $i < \ell < j$ and $v_\ell \in N_L(v_j)$: if $v_i \notin e^{\ell j}$, then $e^{\ell j}_{\le i} = e^{\ell j}_{\leq i-1}$ and $e^{\ell j}_{(i,j)} = e^{\ell j}_{(i-1,j)}$, so $A_{\ell, i} = A_{\ell, i-1}$; if $v_i \in e^{\ell j}$, then $v_i \notin I$ forces $e^{\ell j}_{\le i} \not\subseteq I$, so $A_{\ell, i} = 1$, while $A_{\ell, i-1} \geq 1$.
Thus $A_{\ell, i} \leq A_{\ell, i-1}$ for each $ i < \ell < j$.
Moreover, if $v_i \in N_L(v_j)$, the term $A_{i, i-1}$ is present in $S_{i-1}$ (with value at least $1$) but absent from $S_i$. We conclude that $S_i \le S_{i-1}$ deterministically.
For the remaining cases, we suppose that $v_i \notin B_{i-1}$.

\vspace{5pt}
\noindent\textbf{Case~\ref{role: left-neighbor X}:}
Since $N_L(u) \cap N_L(w) = \varnothing$ for all distinct $u, w \in X \cup \{v_j\}$, there is a unique vertex $u \in X$ with $v_i \in N_L(u)$.
Furthermore, by linearity of $\cH$, there is a unique $i < s < j$ such that $u \in e^{sj}$.
It follows that
\begin{equation*}
    \E[S_i\mid \cF_{i-1}] = \E\left[ p_{i}(v_j)^2\prod_{\substack{i < \ell < j \\ v_\ell \in N_L(v_j)}}A_{\ell,i}\;\Bigg|\; \cF_{i-1}\right] = \E\left[ A_{s,i}\mid \cF_{i-1}\right]\,p_{i-1}(v_j)^2\prod_{\substack{i < \ell < j \\ \ell \neq s \\ v_\ell \in N_L(v_j)}} A_{\ell,i-1}.
\end{equation*}
It suffices to show that $\E\left[ A_{s,i}\mid \cF_{i-1}\right] \leq A_{s,i-1}$.
As $v_i \notin N_L(w)$ for each $w \in e^{sj}\setminus \{u\}$ and $v_i \notin e^{sj}$, we have
\begin{align*}
    \E\left[ A_{s,i}\mid \cF_{i-1}\right] &= 1+(1+\eta)\,\I\left[e^{\ell j}_{\leq i-1}\subseteq I\right] \E[p_{i}(u) \mid \cF_{i-1}]\prod_{w\in e^{s j}_{(i-1,j)} \setminus \{u\}}p_{i-1}(w) \\
    &= 1+(1+\eta)\,\I\left[e^{\ell j}_{\leq i-1}\subseteq I\right] \prod_{w\in e^{s j}_{(i-1,j)}}p_{i-1}(w) = A_{s, i-1},
\end{align*}
by Lemma~\ref{lemma: potential}.

\vspace{5pt}
\noindent\textbf{Case~\ref{role: internal}:}
By the linearity of $\cH$, there is a unique such $s$ with $i < s < j$ and $v_{s} \in N_L(v_j)$ such that $v_i \in e^{s j}$. 
Since Cases~\ref{role: left-neighbor X} and~\ref{role: left-neighbor vj} cannot occur together with Case~\ref{role: internal}, we have
$N_R(v_i)\cap\left(X\cup\{v_j\}\right)=\emptyset$, and thus \(p_i(w)=p_{i-1}(w)\) for every \(w\in X\cup\{v_j\}\).
As 
\[
\I\left[e^{s j}_{\le i}\subseteq I\right]
= \I\left[e^{s j}_{\leq i-1}\subseteq I\right] \cdot 
\I\left[v_i \in I\right],
\]
we have
\begin{align*}
    \E[S_i\mid \cF_{i-1}] &= \E\left[ p_{i}(v_j)^2\prod_{\substack{i < \ell < j \\ v_\ell \in N_L(v_j)}}A_{\ell,i} \;\Bigg|\; \cF_{i-1}\right] \\
    &= \E\left[ A_{s,i}\mid \cF_{i-1}\right]\,p_{i-1}(v_j)^2\prod_{\substack{i < \ell < j \\ \ell \neq s \\ v_\ell \in N_L(v_j)}} A_{\ell,i-1} \\
    &= \left(1 + (1+\eta)\,\E\left[ \I\left[e^{s j}_{\le i} \subseteq I\right]\prod_{u \in e^{s j}_{(i, j)}} p_{i}(u) \;\Bigg|\; \cF_{i-1}\right]\right)\,p_{i-1}(v_j)^2\prod_{\substack{i < \ell < j \\ \ell \neq s \\ v_\ell \in N_L(v_j)}} A_{\ell,i-1}\\
    &= \left(1 + (1+\eta)\,\E\left[\I\left[v_i \in I\right] \,\middle|\, \cF_{i-1}\right] \, \I\left[e^{s j}_{\leq i-1} \subseteq I\right]\prod_{u \in e^{s j}_{(i, j)}} p_{i-1}(u)\right)\,p_{i-1}(v_j)^2\prod_{\substack{i < \ell < j \\ \ell \neq s \\ v_\ell \in N_L(v_j)}} A_{\ell,i-1} \\
    &= p_{i-1}(v_j)^2\prod_{\substack{i < \ell < j  \\ v_\ell \in N_L(v_j)}} A_{\ell,i-1},
\end{align*}
where we use the fact that $\E\left[\I\left[v_i \in I\right]\,\middle|\,\cF_{i-1}\right] = p_{i-1}(v_i)$.

\vspace{5pt}
\noindent\textbf{Case~\ref{role: left-neighbor vj}:}
There are exactly two changes when we go from $S_{i-1}$ to $S_i$: 
$p_i(v_j)$ may differ from $p_{i-1}(v_j)$, and the factor $A_{i, i-1}$ is absent from $S_i$.
Indeed, no other factor changes. For every \(i<\ell<j\) with \(v_\ell\in N_L(v_j)\), since neither Case~\ref{role: left-neighbor X} nor Case~\ref{role: internal} can occur together with Case~\ref{role: left-neighbor vj}, we have \(v_i\notin e^{\ell j}\) and \(v_i\notin N_L(u)\) for every \(u\in e^{\ell j}_{(i,j)}\). Therefore, all weights, indicators, and products appearing in \(A_{\ell,i-1}\) remain unchanged, and hence \(A_{\ell,i}=A_{\ell,i-1}\).
In particular, we have
\[
\E[S_i \mid \cF_{i-1}] = \E\left[ p_{i}(v_j)^2\prod_{\substack{i < \ell < j \\ v_\ell \in N_L(v_j)}}A_{\ell,i} \;\Bigg|\; \cF_{i-1}\right] = \E\left[ p_{i}(v_j)^2\mid \cF_{i-1}\right]\prod_{\substack{i < \ell < j \\ v_\ell \in N_L(v_j)}}A_{\ell,i-1}.
\]
It therefore suffices to show that
\begin{align*}
    \E\left[p_i(v_j)^2\,|\,\cF_{i-1}\right] \le p_{i-1}(v_j)^2A_{i, i-1}
    &= p_{i-1}(v_j)^2 \left(1+(1+\eta)\,\I\left[e^{ij}_{\leq i-1} \subseteq I\right]\prod_{u \in e^{ij}_{(i-1, j)}}p_{i-1}(u)\right) \\
    &= p_{i-1}(v_j)^2 \left(1+(1+\eta)\,\I\left[e^{ij}_{\leq i-1} \subseteq I\right]\,p_{i-1}(v_i)\right),
\end{align*}
where we use the fact that $e^{ij}_{(i-1, j)} = \{v_i\}$.
If $v_j \in B_{i-1}$, its weight is unchanged, so the required bound follows from $A_{i,i-1} \geq 1$. We may therefore assume that $v_j \notin B_{i-1}$.
We claim the following deterministic bound:
\begin{equation}\label{eq: deterministic bounds}
p_i(v_j)^2
\le p_i(v_j)\,\frac{p_{i-1}(v_j)}{1-\I\left[e^{ij}_{\leq i-1} \subseteq I\right]\,p_{i-1}(v_i)}
\le p_i(v_j)\,p_{i-1}(v_j) \left(1+(1+\eta)\,\I\left[e^{ij}_{\leq i-1} \subseteq I\right]\,p_{i-1}(v_i)\right).
\end{equation}
For the first inequality in~\eqref{eq: deterministic bounds}, we must verify that 
\[
p_i(v_j) \le \frac{p_{i-1}(v_j)}{1-\I\left[e^{ij}_{\leq i-1} \subseteq I\right]\,p_{i-1}(v_i)}.
\]
If $e^{ij}_{\leq i-1} \not\subseteq I$, both sides equal $p_{i-1}(v_j)$. If $e^{ij}_{\leq i-1} \subseteq I$, then either  $v_i \in I$ and $p_i(v_j) = 0$ (Step~\ref{step: set zero}), or $v_i \notin I$ and $p_i(v_j) = p_{i-1}(v_j)/(1-p_{i-1}(v_i))$ by Step~\ref{step: weight boost}; in both cases the inequality holds.
For the latter inequality in~\eqref{eq: deterministic bounds}, we use that $\frac{1}{1-x} \le 1+(1+\eta)x$ whenever $0 \le x \le \frac{\eta}{1+\eta}$ with
\[
x = \I\left[e^{ij}_{\leq i-1} \subseteq I\right]\,p_{i-1}(v_i) \le p_{i-1}(v_i) \overset{\eqref{eq:weight-cap}}{\le} \frac{\alpha_{v_i}^{\eta}}{1-\alpha_{v_i}^{\eta}}
\le \frac{\eta}{1+\eta}.
\]
Taking conditional expectations of~\eqref{eq: deterministic bounds} and applying Lemma~\ref{lemma: potential}, we obtain
\begin{align*}
\E\left[p_i(v_j)^2\,|\,\cF_{i-1}\right]
&\le p_{i-1}(v_j) 
\left(1+(1+\eta)\,\I\left[e^{ij}_{\leq i-1} \subseteq I\right]\,p_{i-1}(v_i)\right)
\,\E\left[p_i(v_j)\,|\,\cF_{i-1}\right]\\
&= p_{i-1}(v_j)^2
\left(1+(1+\eta)\,\I\left[e^{ij}_{\leq i-1} \subseteq I\right]\,p_{i-1}(v_i)\right),
\end{align*}
as desired. 

\vspace{5pt}
\noindent\textbf{Case~\ref{role: none}:} 
In this case, $p_i(v_j) = p_{i-1}(v_j)$ and $p_i(u) = p_{i-1}(u)$ for every $u \in X$. Moreover, since $v_i \notin e^{\ell j}$ for every $\ell$ with $i < \ell < j$ and $v_\ell \in N_L(v_j)$, we have $\I\left[e^{\ell j}_{\le i} \subseteq I\right] = \I\left[e^{\ell j}_{\leq i-1} \subseteq I\right]$ for all such $\ell$. Thus, $S_i = S_{i-1}$.
This completes the case analysis.
\end{claimproof}

By Claim~\ref{claim: S_i} and the tower property of conditional expectation, we have
\[
\E[S_{j-1}] = \E\left[\E[S_{j-1}\,|\,\mathcal{F}_{j-2}]\right] \le \E[S_{j-2}] \le \cdots \le \E[S_0].
\]
Since $S_{j-1} = p_{j-1}(v_j)^2$ and $e^{\ell j}_{\le 0} = \varnothing$ for all $\ell$, this implies that
\begin{equation}\label{eq: second moment from supermartingale}
\E[p_{j-1}(v_j)^2] \le \E[S_0] = \alpha_{v_j}^2 \prod_{\substack{1 \le \ell < j \\ v_\ell \in N_L(v_j)}}A_{\ell, 0} = \alpha_{v_j}^2 \prod_{\substack{1 \le \ell < j \\ v_\ell \in N_L(v_j)}}\left(1 + (1+\eta)\prod_{u \in e^{\ell j}_{(0, j)}}p_0(u)\right).
\end{equation}
Note that as each $u \in e^{\ell j}_{(0,j)}$ appears before $v_j$ in the ordering, we have $p_0(u) = \alpha_u \le \alpha_{v_j}$.
Furthermore, since $|N_L(v_j)| \le \deg_{\cH_j}(v_j) \le d^*(v_j)$ (as a result of the linearity of $\cH$),
it follows from~\eqref{eq: second moment from supermartingale} that
\begin{align}
\E[p_{j-1}(v_j)^2]
&\le \alpha_{v_j}^2 \exp\left((1+\eta)\, d^*(v_j) \, \alpha_{v_j}^{k-1}\right)\\
&= \alpha_{v_j}^2\exp\left(\frac{1+\eta}{(1-2\eta/k)^{k-1}} \, d^*(v_j) \, \phi(v_j)^{k-1}\right) \\
&\le \alpha_{v_j}^2\exp\left(\frac{1+\eta}{(1-2\eta/k)^{k-1}(1+\eps)^{k-1}} \, \frac{\log d^*(v_j)}{k-1} \right). \label{eq: variance middle step}
\end{align}
Since $(1-2\eta/k)^{k-1} > 1-2\eta$, we note that
\begin{align}
\frac{1+\eta}{(1-2\eta/k)^{k-1}(1+\eps)^{k-1}} \le \frac{1+\eta}{(1-2\eta)(1+\eps)}
\le \frac{1+4\eta}{1+\eps} \le 1-3\eta,
\end{align}
where the last two inequalities follow from $\eps<1$ and $\eta = \eps/10$.
Substituting this into~\eqref{eq: variance middle step} gives
\begin{align}\label{eq:variance}
\E[p_{j-1}(v_j)^2] \le \alpha_{v_j}^2\exp\left((1-3\eta) \, \frac{\log d^*(v_j)}{k-1} \right) = \alpha_{v_j}^2\,d^*(v_j)^{\frac{1-3\eta}{k-1}}.
\end{align}
By substituting this into~\eqref{eq: v_j in I middle bound}, we obtain that
\begin{equation}
\pr(v_j\in I) \ge \left(1 - 2\alpha_{v_j}^{1-\eta}\,d^*(v_j)^{\frac{1-3\eta}{k-1}}\right)\alpha_{v_j}.
\end{equation}
Since $\alpha_{v_j}^{1-\eta} \le d^*(v_j)^{-\frac{(1-\eta)^2}{k-1}} \le d^*(v_j)^{-\frac{1-2\eta}{k-1}}$ by~\eqref{eq: alpha v_j bound}, it follows that
\begin{equation}
\pr(v_j\in I) \ge \left(1 - 2 d^*(v_j)^{-\frac{\eta}{k-1}}\right)\alpha_{v_j} \ge \left(1 - 2 \eta/k\right)\alpha_{v_j} = \phi(v_j),
\end{equation}
where in the last inequality we again used that $d^*(v_j)$ is sufficiently large.
\end{proof}

\subsection{Proof of Theorem~\ref{thm: local demand}}\label{subsection: local demand prop}

We now strengthen Proposition~\ref{prop: degenerate + local demand}
by allowing the error in the demand bound to decrease with the
local degree.
The proof of the following strengthening of Proposition~\ref{prop: degenerate + local demand} uses the same sampling procedure, with parameters chosen separately for each vertex.

\begin{proposition}\label{prop: degenerate + local demand2}
Let $k \ge 2$ be an integer. There exists a nonincreasing convex function $f_k\colon[0,\infty)\to[0,1]$ satisfying
\[
f_k(x)=(1-o_x(1))\left(\frac{1}{k-1}\,\frac{\log x}{x}\right)^{\frac{1}{k-1}}
\]
such that the following holds for every $n\in\N$.
Let $\cH$ be an $n$-vertex $k$-uniform uncrowded hypergraph with a vertex ordering $(v_1,\dots, v_n)$.
For each $i \in [n]$, let $\cH_i \coloneqq \cH[\{v_1, \ldots, v_i\}]$ and $d^*(v_i) \coloneqq \deg_{\cH_i}(v_i)$. Suppose that there is a demand function $\phi$ for $\cH$ such that:
\begin{enumerate}[(i)]
\item $\phi(v_i)$ is nondecreasing in $i$; and
\item for every $i \in [n]$, $\phi(v_i)\le f_k(d^*(v_i))$.
\end{enumerate}
Then $\cH$ admits a $\phi$-coloring.
Moreover, there is a $\poly(n)$-time algorithm that samples an independent set $I \subseteq V(\cH)$ such that, for each $v \in V(\cH)$, $\pr(v\in I) \ge \phi(v)$.
\end{proposition}

\begin{proof}[Proof sketch]
Consider $h_k \colon (1,\infty) \to \R$, given by 
\[
h_k(x) \coloneqq \left(1-\frac{3}{\sqrt{\log x}}\right)\left(\frac{1}{k-1}\,\frac{\log x}{x}\right)^{\frac{1}{k-1}}.
\]
We note that, for sufficiently large $x$, the function $h_k$ is positive, nonincreasing, and convex. Thus, choosing a sufficiently large constant $C > 1$ and defining $f_k \colon [0,\infty) \to [0,1]$ by $f_k(x) \coloneqq h_k(x+C)$ gives a nonincreasing convex function on $[0,\infty)$ satisfying
\[
f_k(x)=(1-o_x(1))
\left(\frac{1}{k-1} \, \frac{\log x}{x}\right)^{\frac{1}{k-1}}.
\]
Vertices with zero demand may first be deleted, since this can only decrease the degrees in the inherited ordering. For each remaining vertex $v_i$, we define
\[
\eta_i\coloneqq\frac{c_k}{\sqrt{\log(1/\phi(v_i))}} \qquad \text{and} \qquad \alpha_i\coloneqq\frac{\phi(v_i)}{1 - 2\eta_i/k},
\]
where $c_k>0$ is sufficiently small and $C$ is chosen sufficiently large after $c_k$. 
We then run Algorithm~\ref{algorithm: iset uncrowded} with these vertex-dependent parameters.
Since $\phi(v_i)$ is nondecreasing in $i$, the parameters $\eta_i$, $\alpha_i$, and the thresholds $\tau_i\coloneqq\alpha_i^{\eta_i}$ are nondecreasing in $i$, and~\eqref{eq:weight-cap} continues to hold with $\alpha_{v_j}^{\eta}$ replaced by $\tau_j$. For sufficiently large $C$, we also have $(1-\tau_j)^{-1-(k-1)T_k}\le1+\eta_j$, since $\tau_j=\exp\bigl(-\Theta_k(\sqrt{\log(1/\phi(v_j))})\bigr)$. Thus the proof of Claim~\ref{claim: S_i} applies with $\rho_j=\tau_j$ and all edge coefficients bounded by $1+\eta_j$ when analyzing $v_j$.
Moreover, we note that
\[
\eta_i =O\left(\frac{c_k}{\sqrt{\log(d^*(v_i)+C)}}\right).
\]
Thus, by choosing $c_k$ sufficiently small, $\eta_i$ is sufficiently small compared with the error term in the definition of $f_k$.
The remainder of the proof is identical to the proof of Proposition~\ref{prop: degenerate + local demand}, \textit{mutatis mutandis}, and so we omit the details.
\end{proof}

We now deduce Theorem~\ref{thm: local demand} from Proposition~\ref{prop: degenerate + local demand2}.
\begin{proof}[Proof of Theorem~\ref{thm: local demand}]
Let $f_k$ be the function given by Proposition~\ref{prop: degenerate + local demand2}. Consider an ordering $(v_1,\dots,v_n)$ of $V(\cH)$ such that $\phi(v_1)\le\cdots\le\phi(v_n)$.
Since $d^*(v_i)=\deg_{\cH_i}(v_i)\le\deg(v_i)$ and $f_k$ is nonincreasing, the demand condition in Theorem~\ref{thm: local demand} implies
\[
\phi(v_i)\le f_k(\deg(v_i))\le f_k(d^*(v_i)).
\]
Hence, Proposition~\ref{prop: degenerate + local demand2} yields Theorem~\ref{thm: local demand}.
\end{proof}

\section{Lower bound on $\chi_f(\cdot)$ for high-girth $d$-degenerate hypergraphs}\label{section: lower bound}

In this section, we prove Theorem~\ref{thm: degenerate}\ref{item: degenerate lb}. 
The construction is based on the \textit{uniform attachment model} studied in \cite{AcanPittel20, FriezePerezGimenezPralatReiniger19, MalyshkinZhukovskii22,wang2010evolving, barthelemy2022class}.
We formally describe our variant of the model in the following definition.

\begin{definition}\label{def:model}
Let $U_{n,d}^{(k)}$ be the random $k$-uniform hypergraph defined as follows. Start with the vertex set $\{0,\ldots,n\}$, and for each $1\le i\le n$, let $N_i^{-}$ be chosen independently.
If $\binom{i}{k-1}>d$, then $N_i^{-}$ is a uniformly random $d$-element subfamily of $\binom{\{0,\ldots,i-1\}}{k-1}$. Otherwise, $N_i^{-}=\binom{\{0,\ldots,i-1\}}{k-1}$. For every $A\in N_i^{-}$, add the hyperedge $A\cup\{i\}$.
After completing this construction, let \(U_{n,d}^{(k)}\) be the subhypergraph induced by \([n]\coloneqq\{1,\ldots,n\}\); that is, delete vertex \(0\) together with all edges containing it.
\end{definition}

\begin{remark}
It is clear that $U_{n,d}^{(k)}$ is $d$-degenerate with probability one. Indeed, $1,2,\ldots,n$ is a valid degeneracy ordering.
\end{remark}

As discussed in Section~\ref{section: proof overview}, we take advantage of LP duality.
In particular, recalling the linear programming formulation \eqref{eq:chif} for fractional coloring, it suffices to exhibit an appropriate feasible solution to the dual linear program \eqref{eq:fracclique}.
To this end, for $m\ge 0$, let $H_m\coloneqq \sum_{j=1}^m 1/j$, with $H_0\coloneqq 0$. Define the harmonic-tail function $\ell_n:[n]\to\R$ by
\[
\ell_n(i)\coloneqq H_n-H_{i-1}=\sum_{j=i}^n\frac1j,
\qquad
\text{and}
\qquad
\ell_n(S)\coloneqq \sum_{i\in S}\ell_n(i) \quad \text{for } S \subseteq [n].
\]
Changing the order of summation gives
\begin{equation}
\label{eq:total-harmonic-weight}
\ell_n([n])=\sum_{i=1}^n\sum_{j=i}^n\frac1j=n.
\end{equation}

For the remainder of this section, we set
\begin{equation}\label{eq:a-def}
a_{k,d}\coloneqq \left(\frac{1}{k-1}\frac{\log d}{d}\right)^{\frac{1}{k-1}}.
\end{equation}

We state two key lemmas. The first shows that $U_{n,d}^{(k)}$ contains an independent set $I$ with large $\ell_n(I)$ with probability $o(1)$. The second shows that the expected number of short cycles in $U_{n,d}^{(k)}$ is small.

\begin{lemma}\label{lem:heavy-independent}
Fix $k\ge2$ and $0<\varepsilon\leq\frac{1}{4}$. For all sufficiently large $d$,
\[
\pr\left(\exists I\in\mathcal I(U_{n,d}^{(k)}) \text{ such that }
\ell_n(I)\ge(1+\varepsilon)a_{k,d}n\right)=o_{n}(1).
\]
\end{lemma}

\begin{lemma}\label{lem:short-cycles}
For $q\ge2$, let $C_q$ be the number of cycles of length $q$ in $U_{n,d}^{(k)}$. Then
\begin{equation}\label{eq:cycle-bound}
\mathbb E [C_q]
\le 2\cdot 4^{q-2}(k-1)^{2q}d^qH_n.
\end{equation}
\end{lemma}

Before we prove these results, we show how they imply Theorem~\ref{thm: degenerate}\ref{item: degenerate lb}.

\begin{proof}[Proof of Theorem~\ref{thm: degenerate}\ref{item: degenerate lb}]
It suffices to consider $0 < \varepsilon \le \frac{1}{4}$. 
Choose $d$ sufficiently large to satisfy Lemma~\ref{lem:heavy-independent}, and then choose $n$ sufficiently large that all subsequent estimates hold.

By Lemma~\ref{lem:heavy-independent}, with probability at least $3/4$, every independent set $I$ of $U_{n,d}^{(k)}$ satisfies
\begin{equation}\label{eq:no-heavy-independent}
    \ell_n(I) \le (1+\varepsilon)a_{k,d}n.
\end{equation}
On this event, define the weight function
\begin{equation}\label{eq:fractional-clique}
    f(i) \coloneqq \frac{\ell_n(i)}{(1+\varepsilon)a_{k,d}n}, \qquad i \in [n].
\end{equation}
For every independent set $I$ in $U_{n,d}^{(k)}$, we have
\[
\sum_{i \in I} f(i) = \frac{\ell_n(I)}{(1+\varepsilon)a_{k,d}n} \le 1,
\]
which implies that $f$ is a fractional clique in $U_{n,d}^{(k)}$.

For each $2 \le q < g$, let $C_q$ denote the number of $q$-cycles in $U_{n,d}^{(k)}$. Select one vertex from each such cycle, and let $R$ be the set of selected vertices. Define the subhypergraph
\[
\cH \coloneqq U_{n,d}^{(k)}[[n] \setminus R].
\]
Since $R$ hits every cycle of length less than $g$, the hypergraph $\cH$ has girth at least $g$. Furthermore, $\cH$ remains $k$-uniform and $d$-degenerate.

Since $\ell_n(i) \le H_n$ for all $i \in [n]$, the total weight deleted is bounded by
\[
\ell_n(R) \le H_n \sum_{q=2}^{g-1} C_q.
\]
Applying Lemma~\ref{lem:short-cycles} and using \(H_n=O(\log n)\), we obtain
\begin{equation}\label{eq:deleted-weight-expectation}
    \mathbb{E}[\ell_n(R)] \le H_n^2 \sum_{q=2}^{g-1} 2 \cdot 4^{q-2}(k-1)^{2q}d^q \le \frac{\varepsilon^2 n}{4}
\end{equation}
for sufficiently large $n$. Thus, Markov's inequality yields
\begin{equation}\label{eq:small-deletion-event}
    \Pr(\ell_n(R) \le \varepsilon^2 n) \ge 3/4.
\end{equation}

By a union bound, events \eqref{eq:no-heavy-independent} and \eqref{eq:small-deletion-event} hold simultaneously with positive probability. Fix an outcome in their intersection. Since every independent set of $\cH$ is also independent in $U_{n,d}^{(k)}$, the restriction of $f$ to $V(\cH)$ is a fractional clique in $\cH$. Applying weak LP duality together with \eqref{eq:total-harmonic-weight} and \eqref{eq:fractional-clique}, we get
\begin{align*}
    \chi_f(\cH) &\ge \sum_{i \in V(\cH)} f(i) 
    = \frac{\ell_n([n]) - \ell_n(R)}{(1+\varepsilon)a_{k,d}n} \\
    &\ge \frac{1 - \varepsilon^2}{1 + \varepsilon} \, a_{k,d}^{-1} 
    = (1 - \varepsilon) \left( (k-1) \frac{d}{\log d} \right)^{\frac{1}{k-1}}.
\end{align*}

Finally, if $\cH$ has fewer than $n$ vertices, we add isolated vertices until $|V(\cH)| = n$. This preserves $k$-uniformity, $d$-degeneracy, and girth without decreasing the fractional chromatic number, completing the proof.
\end{proof}

We will use the following inequality from \cite{high_girth} to prove Lemma~\ref{lem:heavy-independent}.
\begin{proposition}[{\cite[Proposition~3.4]{high_girth}}]\label{prop:harmonic-energy}
For every $S\subseteq[n]$,
\begin{equation}\label{eq:harmonic-energy}
 \sum_{j\in S}\frac{|S\cap[j]|}{j}
 \ge \frac{\ell_n(S)^2+|S|^2}{2n}.
\end{equation}
\end{proposition}
We next use Proposition~\ref{prop:harmonic-energy} to bound the probability that a fixed set $S\subseteq[n]$ is independent.

\begin{proposition}\label{prop:fixed-set}
Let $\emptyset\ne S\subseteq[n]$, and suppose
\begin{equation}\label{eq:positive-bracket}
 \frac{\ell_n(S)^2+|S|^2}{2n}\ge (k-1)H_n.
\end{equation}
Then
\begin{equation}\label{eq:fixed-set-bound}
 \Pr(S\text{ is independent in }U_{n,d}^{(k)})
 \le
 \exp\left(-d|S|^{2-k}
 \left(
 \frac{\ell_n(S)^2+|S|^2}{2n}-(k-1)H_n
 \right)^{k-1}\right).
\end{equation}
\end{proposition}

\begin{proof}
A set $S \subseteq [n]$ is independent if and only if for every $j \in S$ with $j\ge k$, the family $N_j^{-}$ contains no $(k-1)$-set contained in $S$. 

First, suppose that $\binom{j}{k-1} > d$. There are precisely $\binom{\binom{j}{k-1}}{d}$ choices for the family $N_j^{-} \subseteq \binom{\{0,\ldots,j-1\}}{k-1}$, of which $\binom{\binom{j}{k-1} - \binom{|S \cap [j-1]|}{k-1}}{d}$ contain no $(k-1)$-set from $S$. Consequently,
\begin{align}
\pr\left( N_j^{-} \cap \binom{S \cap [j-1]}{k-1} = \emptyset \right)
&= \frac{\binom{\binom{j}{k-1} - \binom{|S \cap [j-1]|}{k-1}}{d}}{\binom{\binom{j}{k-1}}{d}} = \prod_{r=0}^{d-1} \frac{\binom{j}{k-1} - \binom{|S \cap [j-1]|}{k-1} - r}{\binom{j}{k-1} - r} \nonumber \\
&\le \left( 1 - \frac{\binom{|S \cap [j-1]|}{k-1}}{\binom{j}{k-1}} \right)^d \le \exp\left( -d \frac{\binom{|S \cap [j-1]|}{k-1}}{\binom{j}{k-1}} \right). \label{eq:one-vertex}
\end{align}

We observe that \eqref{eq:one-vertex} holds trivially when $\binom{j}{k-1} \le d$ as well. Indeed, if $|S \cap [j-1]| \ge k-1$, then $N_j^-$ is forced to contain all of $\binom{S \cap [j-1]}{k-1} \neq \emptyset$, making the left-hand side of \eqref{eq:one-vertex} zero. Conversely, if $|S \cap [j-1]| < k-1$, then $\binom{|S \cap [j-1]|}{k-1} = 0$, so both sides equal one. Thus, \eqref{eq:one-vertex} holds for all $k \le j \le n$.

Since the families $N_j^{-}$ are selected independently across $j$, it follows from \eqref{eq:one-vertex} that
\begin{equation}\label{eq:independent-product}
\Pr(S \text{ is independent}) \le \exp\left( -d \sum_{\substack{j \in S \\ j \ge k}} \frac{\binom{|S \cap [j-1]|}{k-1}}{\binom{j}{k-1}} \right).
\end{equation}

For every $j \in S$ with \(j\geq k\), we have the bound
\begin{equation}\label{eq:binomial-ratio}
\frac{\binom{|S \cap [j-1]|}{k-1}}{\binom{j}{k-1}} \ge \left( \frac{\max\{0, \, |S \cap [j-1]| - k + 2\}}{j} \right)^{k-1}.
\end{equation}
Indeed, if $|S \cap [j-1]| < k-1$, both sides of \eqref{eq:binomial-ratio} are zero. Otherwise, 
\[
\frac{\binom{|S \cap [j-1]|}{k-1}}{\binom{j}{k-1}} = \frac{|S \cap [j-1]|(|S \cap [j-1]|-1)\cdots(|S \cap [j-1]|-k+2)}{j(j-1)\cdots(j-k+2)} \ge \left( \frac{|S \cap [j-1]| - k + 2}{j} \right)^{k-1}.
\]
Applying \eqref{eq:binomial-ratio} alongside Jensen's inequality yields
\begin{align}
\sum_{\substack{j \in S \\ j \ge k}} \frac{\binom{|S \cap [j-1]|}{k-1}}{\binom{j}{k-1}} & \ge \sum_{\substack{j \in S \\ j \ge k}}\left( \frac{\max\{0, \, |S \cap [j-1]| - k + 2\}}{j} \right)^{k-1}
=\sum_{j \in S}\left( \frac{\max\{0, \, |S \cap [j-1]| - k + 2\}}{j} \right)^{k-1}\\
&\ge |S| \left( \frac{1}{|S|} \sum_{j \in S} \frac{\max\{0, \, |S \cap [j-1]| - k + 2\}}{j} \right)^{k-1}. \label{eq:power-mean}
\end{align}

Since $\max\{0, t\} \ge t$ for all real $t$, the inner sum is bounded below by
\begin{align}
\sum_{j \in S} \frac{\max\{0, \, |S \cap [j-1]| - k + 2\}}{j} 
&\ge \sum_{j \in S} \frac{|S \cap [j-1]| - k + 2}{j} = \sum_{j \in S} \frac{|S \cap [j]| - (k-1)}{j} \nonumber \\
&\ge \frac{\ell_n(S)^2 + |S|^2}{2n} - (k-1)H_n, \label{eq:sum-lower-bound}
\end{align}
where the final step follows from Proposition~\ref{prop:harmonic-energy}.

By \eqref{eq:positive-bracket}, the expression in \eqref{eq:sum-lower-bound} is non-negative. Combining \eqref{eq:independent-product}, \eqref{eq:power-mean}, and \eqref{eq:sum-lower-bound}, we conclude that
\[
\pr\left(S~\text{is independent in}~ U_{n,d}^{(k)}\right) \le \exp\left( -d |S|^{2-k} \left( \frac{\ell_n(S)^2 + |S|^2}{2n} - (k-1)H_n \right)^{k-1} \right),
\]
as desired.
\end{proof}

We are now ready to prove Lemma~\ref{lem:heavy-independent}.

\begin{proof}[Proof of Lemma~\ref{lem:heavy-independent}]
Let
\[
L \coloneqq \log\left( \frac{1}{(1+\varepsilon)a_{k,d}} \right).
\]
Since $a_{k,d} \coloneqq \left(\frac{1}{k-1}\, \frac{\log d}{d}\right)^{\frac{1}{k-1}}$, we have
\[
L = \frac{\log d}{k-1} + 
O(\log\log d).
\]
For $d$ sufficiently large, we may assume that $L \ge 4$ and
\begin{equation}\label{eq:entropy-coefficient}
L + 1 + \log(2L) \le \left(1 + \frac{\varepsilon}{2}\right) \frac{\log d}{k-1}.
\end{equation}

Let $S \subseteq [n]$ satisfy $\ell_n(S) \ge (1+\varepsilon)a_{k,d}n$, and set \[s \coloneqq |S|,\qquad \rho \coloneqq s/n, \qquad \text{and} \qquad y \coloneqq \ell_n(S)/n.\] Since $\ell_n(i)$ is decreasing in $i$, the sum $\ell_n(S)$ among sets of size $s$ is maximized by $[s]$. Hence,
\begin{equation}\label{eq:weight-size}
y \le \frac{\ell_n([s])}{n} = \frac{s}{n}(1 + H_n - H_s) \le \rho \left(1 + \log(1/\rho)\right).
\end{equation}

We claim that
\begin{equation}\label{eq:size-lower}
\rho \ge \frac{(1+\varepsilon)a_{k,d}}{2L}.
\end{equation}
To see this, set $\rho_0 \coloneqq \frac{(1+\varepsilon)a_{k,d}}{2L}$. By the definition of $L$ and the assumption $L \ge 4$, we have
\[
\rho_0 \left(1 + \log(1/\rho_0)\right) = (1+\varepsilon)a_{k,d} \frac{1 + L + \log(2L)}{2L} < (1+\varepsilon)a_{k,d}.
\]
Because $t \mapsto t\left(1 + \log(1/t)\right)$ is strictly increasing on $(0, 1]$, if $\rho < \rho_0$ were to hold, then \eqref{eq:weight-size} would yield
\[
y \le \rho \left(1 + \log(1/\rho)\right) < \rho_0 \left(1 + \log(1/\rho_0)\right) < (1+\varepsilon)a_{k,d},
\]
contradicting our choice of $S$. Thus $\rho \ge \rho_0$, establishing \eqref{eq:size-lower}.

Now choose $n$ sufficiently large that
\begin{equation}\label{eq:harmonic-error}
\frac{(k-1)H_n}{n} \le \frac{\varepsilon}{8} \, \rho (1+\varepsilon)a_{k,d}
\end{equation}
for all $\rho$ satisfying \eqref{eq:size-lower}. This choice is possible because the right-hand side is bounded below by a positive constant depending only on $k, d, \varepsilon$, whereas $H_n/n \to 0$ as $n \to \infty$.

Since $y \ge (1+\varepsilon)a_{k,d}$ by the choice of \(S\), we obtain
\begin{equation}\label{eq:energy-lower}
\frac{\rho^2 + y^2}{2} - \frac{(k-1)H_n}{n} \ge \rho y - \frac{\varepsilon}{8}\rho y = \left(1 - \frac{\varepsilon}{8}\right)\rho y.
\end{equation}
Thus, the hypotheses of Proposition~\ref{prop:fixed-set} are satisfied, implying that
\begin{align}
\Pr(S \text{ is independent}) 
&\le \exp\left( -nd \rho^{2-k} \left( \frac{\rho^2 + y^2}{2} - \frac{(k-1)H_n}{n} \right)^{k-1} \right) \nonumber \\
&\le \exp\left( -nd \left(1 - \frac{\varepsilon}{8}\right)^{k-1} \rho y^{k-1} \right) \nonumber \\
&\le \exp\left( -s \left( \left(1 - \frac{\varepsilon}{8}\right)(1+\varepsilon) \right)^{k-1} \frac{\log d}{k-1} \right). \label{eq:fixed-heavy}
\end{align}
For $0 < \varepsilon \le 1/4$, we have
\[
\left(1 - \frac{\varepsilon}{8}\right)(1+\varepsilon) = 1 + \frac{7\varepsilon}{8} - \frac{\varepsilon^2}{8} \ge 1 + \frac{3\varepsilon}{4}.
\]
Therefore,
\begin{equation}\label{eq:fixed-heavy-simple}
\Pr(S \text{ is independent}) \le \exp\left( -\left(1 + \frac{3\varepsilon}{4}\right) \frac{s\log d}{k-1} \right).
\end{equation}

It remains to take a union bound over all candidate sets $S$. Using standard binary entropy bounds, we have $\binom{n}{s} \le e^{n h(\rho)}$, where
\[
h(\rho) \coloneqq \rho \log(1/\rho) + (1-\rho)\log\left(\frac{1}{1-\rho}\right) \le \rho \left(1 + \log(1/\rho)\right).
\]
By \eqref{eq:size-lower}, $\log(1/\rho) \le L + \log(2L)$, so
\[
h(\rho) \le \rho \left(L + 1 + \log(2L)\right).
\]
Applying \eqref{eq:entropy-coefficient}, we obtain
\[
\binom{n}{s} \le \exp\left( \left(1 + \frac{\varepsilon}{2}\right) \frac{s\log d}{k-1} \right).
\]

Combining this with \eqref{eq:fixed-heavy-simple}, the expected number of independent sets $S$ of size $s$ satisfying $\ell_n(S) \ge (1+\varepsilon)a_{k,d}n$ is
\[
\sum_{\substack{S \subseteq [n],\ |S| = s \\ \ell_n(S) \ge (1+\varepsilon)a_{k,d}n}} \Pr(S \text{ is independent}) \le \exp\left( -\frac{\varepsilon}{4(k-1)} s \log d \right).
\]
Summing over all $s \ge \frac{(1+\varepsilon)a_{k,d}}{2L}n$ via \eqref{eq:size-lower}, we conclude that
\begin{align*}
\pr\left( \exists I \in \mathcal{I}(U_{n,d}^{(k)})~\text{such that}~\ell_n(I) \ge (1+\varepsilon)a_{k,d}n \right) 
&\ \le \sum_{s \ge \frac{(1+\varepsilon)a_{k,d}}{2L} n} \exp\left( -\frac{\varepsilon}{4(k-1)} s \log d \right) \\
&\ = o_n(1),
\end{align*}
completing the proof.
\end{proof}

Finally, we present the proof of Lemma \ref{lem:short-cycles}.

\begin{proof}[Proof of Lemma \ref{lem:short-cycles}]
It is convenient first to retain vertex $0$; denote the resulting random hypergraph by $\widetilde{U}^{(k)}_{n,d}$. 
Let $e \subseteq \{0, \ldots, n\}$ be a $k$-set whose maximum vertex is $j \in [n]$, and define the weight function $w(e) \coloneqq \frac{d}{\binom{j}{k-1}}$.

If $e_1, \ldots, e_t$ are distinct $k$-sets of this form, then
\begin{equation}\label{eq:joint-edge-bound}
\pr\left( e_1, \ldots, e_t \in E(\widetilde{U}^{(k)}_{n,d}) \right) \le \prod_{r=1}^t w(e_r).
\end{equation}
Indeed, if $m$ of the prescribed edges have maximum vertex $j$, then whenever $\binom{j}{k-1} > d$, the probability that all corresponding lower $(k-1)$-sets are selected is
\[
\frac{(d)_m}{\left( \binom{j}{k-1} \right)_m} \le \left( \frac{d}{\binom{j}{k-1}} \right)^m,
\]
where $(z)_m = z(z-1)\cdots(z-m+1)$ denotes the falling factorial. If $\binom{j}{k-1} \le d$, the right-hand side is at least one. Since edge selections at distinct maximum vertices are mutually independent, \eqref{eq:joint-edge-bound} follows.

For distinct $a, b \in [n]$, define the kernel sum
\[
K(a,b) \coloneqq \sum_{\substack{e \subseteq \{0, \ldots, n\} \\ |e| = k, \, a,b \in e}} w(e),
\]
and set $K(a,a) \coloneqq 0$. Let $m \coloneqq \max\{a,b\}$. The contribution to $K(a,b)$ from edges whose maximum vertex is $m$ is zero if $m<k-1$, and otherwise is at most
\[
\binom{m-1}{k-2} \frac{d}{\binom{m}{k-1}} = \frac{(k-1)d}{m}.
\]
For $j>m$, this contribution is zero if $k=2$ or $j<k-1$; otherwise it is at most
\[
\binom{j-2}{k-3} \frac{d}{\binom{j}{k-1}} = \frac{(k-1)(k-2)d}{j(j-1)},
\]
In all cases, the contributions are bounded by the respective expressions on the right.\footnote{In particular, the second contribution is zero when $k=2$ and does not appear in the argument of \cite[Lemma~3.3]{high_girth}.} Telescoping the sum over $j>m$, we obtain
\begin{align}\label{eq:pair-kernel}
K(a,b) &\le \frac{(k-1)d}{m} + (k-1)(k-2)d \sum_{j=m+1}^{n} \frac{1}{j(j-1)} \nonumber \\
&= \frac{(k-1)d}{m} + (k-1)(k-2)d \sum_{j=m+1}^{n} \left( \frac{1}{j-1} - \frac{1}{j} \right) \nonumber \\
&= \frac{(k-1)d}{m} + (k-1)(k-2)d \left( \frac{1}{m} - \frac{1}{n} \right) \nonumber \\
&\le \frac{(k-1)^2 d}{m} = \frac{(k-1)^2 d}{\max\{a,b\}}.
\end{align}

Every $q$-cycle yields an ordered sequence of distinct core vertices $v_1, \ldots, v_q \in [n]$ and distinct edges $e_1, \ldots, e_q$ such that $\{v_i, v_{i+1}\} \subseteq e_i$ for $1 \le i \le q$ (with indices taken modulo $q$). Applying \eqref{eq:joint-edge-bound} and dropping the distinctness conditions except $v_i\ne v_{i+1}$ for every $i$ gives
\begin{align}\label{eq:trace-reduction}
\mathbb{E}[C_q] &\leq
\sum_{\substack{v_1,\ldots,v_q\in[n]\\v_i\ne v_{i+1}\ \forall i}}
\ \sum_{\substack{e_1,\ldots,e_q\\
                  |e_i|=k,\ \{v_i,v_{i+1}\}\subseteq e_i\ \forall i}}
\prod_{i=1}^q w(e_i)=
\sum_{\substack{v_1,\ldots,v_q\in[n]\\v_i\ne v_{i+1}\ \forall i}}
\prod_{i=1}^q
\left(
\sum_{\substack{e_i\subseteq\{0,\ldots,n\}\\
                |e_i|=k,\ \{v_i,v_{i+1}\}\subseteq e_i}}
w(e_i)
\right)\\
&= \sum_{v_1, \ldots, v_q \in [n]} \prod_{i=1}^q K(v_i, v_{i+1}) 
\le \left((k-1)^2 d\right)^q \sum_{v_1, \ldots, v_q \in [n]} \prod_{i=1}^q \frac{1}{\max\{v_i, v_{i+1}\}} = \left((k-1)^2 d\right)^q \operatorname{Tr}(A^q),
\end{align}
where the second equality uses $K(v,v)=0$, and $A$ is the $n \times n$ real symmetric matrix defined by $A_{ij} \coloneqq \frac{1}{\max\{i,j\}}$.

We claim that $\Tr(A^q) \le 2 \cdot 4^{q-2} H_n$.
We note that the bound $\Tr(A^q) \leq \binom{2q-2}{q-1}H_n$ was established in the proof of \cite[Lemma~3.3]{high_girth} via an intricate combinatorial analysis relying on Catalan numbers.
Here, we present an alternative and considerably simpler proof of this slightly weaker bound via matrix-norm analysis.
To bound the spectral radius $\|A\|_2$, we employ Schur's test~\cite{schur1911bemerkungen} (see also~\cite[Theorem~8.3.2]{heil2018metrics}) with the test vector $u \in \mathbb{R}_+^n$ given by $u_i = i^{-1/2}$. For every $i \in [n]$, we have
\[
\sum_{j=1}^n A_{ij} u_j = \frac{1}{i} \sum_{j=1}^i j^{-1/2} + \sum_{j=i+1}^n j^{-3/2}.
\]
Induction easily yields $\sum_{j=1}^i j^{-1/2} \le 2\sqrt{i}$, while integral comparison gives $\sum_{j=i+1}^n j^{-3/2} \le \int_i^\infty x^{-3/2} \, dx = \frac{2}{\sqrt{i}}$. Summing these bounds, we get
\[
\sum_{j=1}^n A_{ij} u_j \le \frac{4}{\sqrt{i}} = 4 u_i.
\]
By Schur's test, $\|A\|_2 \le 4$.

Letting $\lambda_1, \ldots, \lambda_n$ denote the real eigenvalues of $A$, we have
\[
\operatorname{Tr}(A^q) \le \sum_{r=1}^n |\lambda_r|^q \le \|A\|_2^{q-2} \sum_{r=1}^n \lambda_r^2 \le 4^{q-2} \operatorname{Tr}(A^2).
\]
Finally, using $\sum_{j=i+1}^n j^{-2} \le \int_i^\infty x^{-2} \, dx = \frac{1}{i}$, we calculate
\[
\operatorname{Tr}(A^2) = \sum_{i=1}^n \sum_{j=1}^n \frac{1}{\max\{i,j\}^2} = \sum_{i=1}^n \left( \frac{1}{i} + \sum_{j=i+1}^n \frac{1}{j^2} \right) \le 2 \sum_{i=1}^n \frac{1}{i} = 2 H_n.
\]
Thus, $\operatorname{Tr}(A^q) \le 2 \cdot 4^{q-2} H_n$. Substituting this bound into \eqref{eq:trace-reduction} yields
\[
\mathbb{E}[C_q] \le 2 \cdot 4^{q-2} (k-1)^{2q} d^q H_n,
\]
which completes the proof.
\end{proof}

\section{Concluding remarks}\label{section: concluding remarks}

As mentioned in \S\ref{subsection: main results}, our analysis of the upper-bound results (Theorems~\ref{thm: degenerate}\ref{item: degenerate ub} and~\ref{thm: local demand}) does not strictly require uncrowdedness.
Indeed, the only place in the analysis of Algorithm~\ref{algorithm: iset uncrowded} that uses the absence of $3$- and $4$-cycles is the proof of Claim~\ref{claim: S_i}.
For a linear hypergraph, the proof of that claim, and hence of Theorems~\ref{thm: degenerate}\ref{item: degenerate ub} and~\ref{thm: local demand}, goes through whenever the following cycle configurations are absent with respect to the given vertex ordering:
\begin{enumerate}[label=(Conf\arabic*)]
    \item\label{item: 3 cycle type} A $3$-cycle whose edges can be labeled $e, f, g$ so that the rightmost vertex $v$ of $e \cup f \cup g$ lies in $e$, and the rightmost vertex of at least one of $f$ and $g$ also lies in $e$; see Fig.~\ref{fig:configuration-A} for an illustration.

    \item\label{item: 4 cycle type} A $4$-cycle with edges $e, g, h, f$ in cyclic order such that the rightmost vertex $v$ of their union lies in $e \cap f$, and the rightmost vertices of $g$ and $h$ lie in $e$ and $f$, respectively; see Fig.~\ref{fig:configuration-B} for an illustration.

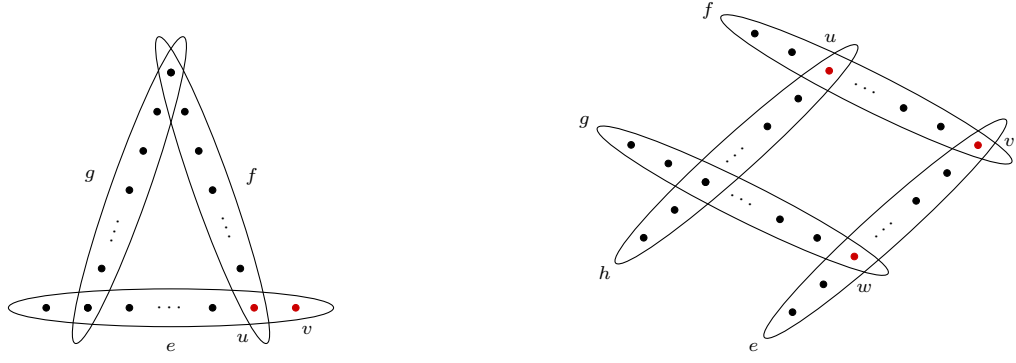
\begin{figure}[htb!]
    \centering

\begin{subfigure}[t]{0.48\textwidth}
    \centering

\begin{tikzpicture}[
    scale=1.0,
    every node/.style={font=\fontsize{12}{14}\selectfont}
]

\tikzset{
    vert/.style={
        circle,
        fill=black,
        draw=black,
        line width=0.15pt,
        inner sep=0pt,
        minimum size=2.5pt
    },
    rvert/.style={
        circle,
        fill=red!80!black,
        draw=red!80!black,
        line width=0.15pt,
        inner sep=0pt,
        minimum size=2.5pt
    }
}

\coordinate (a) at (-1.100,0.000);
\coordinate (u) at (1.100,0.000);
\coordinate (b) at (0.000,3.111);
\coordinate (v) at (1.650,0.000);

\draw
    (0.000,0.000) ellipse (2.15 and 0.25);

\node[font=\scriptsize] at (0.000,-0.520) {$e$};

\node[vert] at (-1.650,0.000) {};
\node[vert] at (a) {};
\node[vert] at (-0.550,0.000) {};
\node[
    font=\scriptsize
] at (0.000,0.000) {$\cdots$};
\node[vert] at (0.550,0.000) {};
\node[rvert] at (u) {};
\node[rvert] at (v) {};

\node[font=\scriptsize] at (0.950,-0.400) {$u$};
\node[font=\scriptsize] at (1.800,-0.300) {$v$};

\draw[
    rotate around={70.529:(-0.550,1.556)}
]
    (-0.550,1.556) ellipse (2.15 and 0.25);

\node[font=\scriptsize] at (-1.069,1.739) {$g$};

\node[vert] at (a) {};
\node[vert] at (-0.917,0.519) {};
\node[
    font=\scriptsize,
    rotate=70.529
] at (-0.733,1.037) {$\cdots$};
\node[vert] at (-0.550,1.556) {};
\node[vert] at (-0.367,2.074) {};
\node[vert] at (-0.183,2.593) {};
\node[vert] at (b) {};

\draw[
    rotate around={109.471:(0.550,1.556)}
]
    (0.550,1.556) ellipse (2.15 and 0.25);

\node[font=\scriptsize] at (1.069,1.739) {$f$};

\node[vert] at (0.917,0.519) {};
\node[
    font=\scriptsize,
    rotate=-70.529
] at (0.733,1.037) {$\cdots$};
\node[vert] at (0.550,1.556) {};
\node[vert] at (0.367,2.074) {};
\node[vert] at (0.183,2.593) {};
\node[vert] at (b) {};

\end{tikzpicture}

\caption{
A forbidden $3$-cycle configuration. Here, $v$ is the rightmost vertex of $e\cup f\cup g$,
and $u$ is the rightmost vertex of $f$.
Note that we may have $u=v$.
}
\label{fig:configuration-A}
\end{subfigure}
    \hfill
\begin{subfigure}[t]{0.48\textwidth}
    \centering

\begin{tikzpicture}[
    scale=1.0,
    every node/.style={font=\fontsize{12}{14}\selectfont}
]

\tikzset{
    vert/.style={
        circle,
        fill=black,
        draw=black,
        line width=0.15pt,
        inner sep=0pt,
        minimum size=2.5pt
    },
    rvert/.style={
        circle,
        fill=red!80!black,
        draw=red!80!black,
        line width=0.15pt,
        inner sep=0pt,
        minimum size=2.5pt
    }
}

\coordinate (x) at (0.000,0.000);
\coordinate (w) at (1.968,-0.984);
\coordinate (u) at (1.635,1.472);
\coordinate (v) at (3.603,0.488);

\draw[
    rotate around={-26.565:(0.492,-0.246)}
]
    (0.492,-0.246) ellipse (2.15 and 0.25);

\node[font=\scriptsize] at (-1.603,0.801) {$g$};

\node[vert] at (-0.984,0.492) {};
\node[vert] at (-0.492,0.246) {};
\node[vert] at (x) {};
\node[
    font=\scriptsize,
    rotate=-26.565
] at (0.492,-0.246) {$\cdots$};
\node[vert] at (0.984,-0.492) {};
\node[vert] at (1.476,-0.738) {};
\node[rvert] at (w) {};

\node[font=\scriptsize] at (2.100,-1.400) {$w$};

\draw[
    rotate around={41.987:(0.409,0.368)}
]
    (0.409,0.368) ellipse (2.15 and 0.25);

\node[font=\scriptsize] at (-1.332,-1.198) {$h$};

\node[vert] at (-0.818,-0.736) {};
\node[vert] at (-0.409,-0.368) {};
\node[vert] at (x) {};
\node[
    font=\scriptsize,
    rotate=41.987
] at (0.409,0.368) {$\cdots$};
\node[vert] at (0.818,0.736) {};
\node[vert] at (1.226,1.104) {};
\node[rvert] at (u) {};

\node[font=\scriptsize] at (1.650,1.900) {$u$};

\draw[
    rotate around={-26.565:(2.127,1.226)}
]
    (2.127,1.226) ellipse (2.15 and 0.25);

\node[font=\scriptsize] at (0.032,2.273) {$f$};

\node[vert] at (0.651,1.964) {};
\node[vert] at (1.143,1.718) {};
\node[
    font=\scriptsize,
    rotate=-26.565
] at (2.127,1.226) {$\cdots$};
\node[vert] at (2.619,0.980) {};
\node[vert] at (3.111,0.734) {};

\draw[
    rotate around={41.987:(2.377,-0.616)}
]
    (2.377,-0.616) ellipse (2.15 and 0.25);

\node[font=\scriptsize] at (0.636,-2.182) {$e$};

\node[vert] at (1.150,-1.720) {};
\node[vert] at (1.559,-1.352) {};
\node[
    font=\scriptsize,
    rotate=41.987
] at (2.377,-0.616) {$\cdots$};
\node[vert] at (2.785,-0.248) {};
\node[vert] at (3.194,0.120) {};
\node[rvert] at (v) {};

\node[font=\scriptsize] at (4.020,0.530) {$v$};

\end{tikzpicture}
\caption{
A forbidden $4$-cycle configuration. Here, $v$ is the rightmost vertex of $e\cup f\cup g\cup h$,
while $u$ and $w$ are the rightmost vertices of $h$ and $g$,
respectively.
}
\label{fig:configuration-B}
\end{subfigure}

\caption{The forbidden cycle configurations.}

\end{figure}

\end{enumerate}
For Theorem~\ref{thm: degenerate}\ref{item: degenerate ub}, this must be a degeneracy ordering, whereas for Theorem~\ref{thm: local demand}, it must be nondecreasing with respect to the prescribed demands.
Notably, for a $d$-degenerate linear hypergraph, each vertex is the rightmost vertex in ``few'' such cycles under any degeneracy ordering.
Therefore, combining this relaxation with random sampling in the style of Duke, Lefmann, and R\"odl~\cite{duke1995uncrowded} yields the following corollary.

\begin{corollary}\label{corollary: degenerate linear}
    For all $\eps > 0$ and integers $k \ge 3$, there exists $d_0 \in \N$ such that the following holds for all $d \ge d_0$ and $n \in \N$. Let $\cH$ be an $n$-vertex $k$-uniform $d$-degenerate linear hypergraph. Then
    \[
    \chi_f(\cH) \le (1+\eps)\left(\frac{(4k-5)(k-1)}{k-2} \, \frac{d}{\log d}\right)^{\frac{1}{k-1}}.
    \]
    Moreover, there is a $\mathrm{poly}(n, d)$-time algorithm that samples an independent set $I \subseteq V(\cH)$ such that, for each $v \in V(\cH)$,
    \[
    \pr(v \in I) \ge (1-\eps)\left(\frac{k-2}{(4k-5)(k-1)} \, \frac{\log d}{d}\right)^{\frac{1}{k-1}}.
    \]
\end{corollary}
The proof of Corollary~\ref{corollary: degenerate linear} is technical but rather standard, so we omit the formal details and instead provide a sketch in Appendix~\ref{appendix}.

It is natural to ask whether the bound in Theorem~\ref{thm: degenerate}\ref{item: degenerate ub} extends to all linear hypergraphs; that is, whether we can remove the restriction on cycle configurations.

\begin{problem}\label{problem: degeneracy linear}
Let $k \ge 3$ be an integer, and let $\eps > 0$ be arbitrary. Does there exist $d_0 \in \N$ such that the following holds for every $d \ge d_0$? If $\cH$ is a $k$-uniform $d$-degenerate linear hypergraph, then
\[
\chi_f(\cH) \le (1+\eps)\left((k-1) \frac{d}{\log d}\right)^{\frac{1}{k-1}}.
\]
\end{problem}

Theorem~\ref{thm: local demand} gives a sharp vertex-dependent guarantee for uncrowded hypergraphs, whereas Corollary~\ref{corollary: degenerate linear} applies to the broader class of linear hypergraphs but gives the same inclusion guarantee at every vertex.
It is natural to ask whether the local statement remains true under the weaker assumption of linearity.

\begin{problem}\label{problem: local linear demands}
Let $k \ge 3$ be an integer, and let $\eps > 0$ be arbitrary. Does there exist $\delta_0 \in \N$ such that the following holds? If $\cH$ is a $k$-uniform linear hypergraph of minimum degree at least $\delta_0$, and $\phi$ is a demand function for $\cH$ satisfying, for every $v \in V(\cH)$,
\[
\phi(v) \le (1-\eps)\left(\frac{1}{k-1} \, \frac{\log \deg(v)}{\deg(v)}\right)^{\frac{1}{k-1}},
\]
must $\cH$ admit a $\phi$-coloring?
\end{problem}

A positive answer to Problem~\ref{problem: local linear demands} would extend the fixed-error consequence of Theorem~\ref{thm: local demand} from uncrowded hypergraphs to linear hypergraphs. Even for regular linear hypergraphs, it would strengthen Corollary~\ref{corollary: degenerate linear} to the asymptotically optimal bound in Theorem~\ref{thm: degenerate}\ref{item: degenerate ub}. The random-sampling reduction used for Corollary~\ref{corollary: degenerate linear} first passes to an induced subhypergraph without the forbidden cycle configurations. When the vertex degrees vary widely, a single sampling probability does not preserve the scale required by every vertex.

One possible direction is to avoid this initial sampling step and instead expose the vertices in a random order, discarding or postponing a vertex only when its exposure would complete a short cycle. The central difficulty is to eliminate all cycles of length at most four while retaining the required inclusion probability separately at every vertex. Resolving this difficulty would answer Problem~\ref{problem: local linear demands}.

\medskip
\noindent
\textit{Forthcoming work on triangle-free hypergraphs.}
As defined in Section~\ref{subsection: background}, a \emph{triangle} is a collection of three distinct edges $e,f,g$ and three distinct vertices $u,v,w$ such that $u\in e\cap f$, $v\in f\cap g$, $w\in e\cap g$, and $\{u,v,w\}\cap e\cap f\cap g=\varnothing$. Note that a triangle-free hypergraph need not be linear. Cooper and Mubayi~\cite[Theorem~2]{cooper2015list} proved that $\chi_\ell(\cH)=O\left(\sqrt{\Delta/\log\Delta}\right)$ for triangle-free $3$-uniform hypergraphs of maximum degree at most $\Delta$. Extending this result, Li and Postle~\cite[Theorem~1.4]{lipostle2022} proved that $\chi_\ell(\cH)=O_k\left((\Delta/\log\Delta)^{\frac{1}{k-1}}\right)$ for every fixed $k\geq3$ and every triangle-free $k$-uniform hypergraph of maximum degree at most $\Delta$.
In forthcoming work~\cite{dhawanli2026trianglefree}, Dhawan, Li, Methuku, and Vo generalize this result to fractional coloring in terms of degeneracy, resolving a conjecture of Dhawan~\cite[Conjecture~5.4]{dhawan2026fractional}.
\begin{theorem}[\cite{dhawanli2026trianglefree}]\label{thm: forthcoming triangle-free}
For every integer $k\geq2$, there exists $d_0\in\N$ such that, for every $d\geq d_0$, every $k$-uniform $d$-degenerate triangle-free hypergraph $\cH$ satisfies
\[
 \chi_f(\cH)\leq (1+o_d(1))\left(\frac{k}{k-1}\right)\left(k(k-1)\,\frac{d}{\log d}\right)^{\frac{1}{k-1}},
\]
where $k$ is fixed as $d\to\infty$.
\end{theorem}

Setting $k=2$ in the above result recovers Theorem~\ref{theo: martinsson steiner}. For general $k$, it improves the bound in Theorem~\ref{theo: hypergraph girth 4} by a factor of $(k-1)^{\frac{1}{k-1}}$ and extends it to the more general setting of triangle-free hypergraphs.
Since $\alpha(\cH)\geq |V(\cH)|/\chi_f(\cH)$, Theorem~\ref{thm: forthcoming triangle-free} also yields a corresponding bound on the independence number of $d$-degenerate triangle-free hypergraphs.
By applying random sampling as in~\cite[Sections~1.2 and~1.3]{lipostle2022}, we obtain improved bounds on the independence number for several classes of sparse hypergraphs.

After completing a preliminary draft of our forthcoming paper~\cite{dhawanli2026trianglefree}, in which we prove Theorem~\ref{thm: forthcoming triangle-free}, we learned that Yu and Zhang~\cite{yu2026trianglefree} had independently obtained a proof of Theorem~\ref{thm: forthcoming triangle-free}. Their proof differs from ours. As in their proof of the upper bound in Theorem~\ref{thm: degenerate}, they use an independent set sampling approach in the style of Martinsson and Steiner~\cite{martinsson2025random}. Our forthcoming paper, by contrast, uses the techniques introduced in this paper both to prove this result and to develop further applications.

\subsection*{Statement of AI use}

During the development of the upper-bound argument, the authors first formulated a simplified version of the procedure in~\cite{dhawan2026fractional}, eliminating auxiliary coin flips and lower-threshold truncation, and replacing the original weight-update rule with a simpler survival-normalization rule: a weight is set to zero if the vertex does not survive and is otherwise divided by its survival probability. This preliminary procedure retained a parameter $\theta$ controlling the activation of unprocessed vertices. Using OpenAI's ChatGPT 5.6 Sol to optimize the choice of $\theta$ suggested setting $\theta=0$, which eliminated the need for auxiliary activations and led to the simplified procedure analyzed in this paper.

Additionally, ChatGPT 5.6 Sol was used to polish the writing, improve the exposition, and support literature searches. All suggestions were thoroughly reviewed and refined by the authors, who take full responsibility for the correctness of the arguments, citations, and final exposition.

\vspace{2mm}
\printbibliography

\appendix

\section{Proof of Corollary~\ref{corollary: degenerate linear}}\label{appendix}

In this section, we give an informal sketch of the proof of Corollary~\ref{corollary: degenerate linear}.
We adapt the random sampling approach of Duke, Lefmann, and R\"odl~\cite{duke1995uncrowded}---which reduces the linear setting to the uncrowded one---to the degeneracy setting, where it suffices to exclude the configurations above.
It suffices to consider $0<\eps<1$.
Fix a $d$-degenerate $k$-uniform linear hypergraph $\cH$ and a degeneracy ordering $(v_1, \ldots, v_n)$ of its vertices. Let $A$ be a $p$-random subset of $V(\cH)$, where
\[
p \coloneqq d^{-\frac{1-\mu}{k-1}},
\qquad
\mu \coloneqq (1-\eps/4)^{k-1}\frac{k-2}{4k-5}.
\]
Set $\cH' \coloneqq \cH[A]$, with the ordering inherited from the fixed degeneracy ordering of $\cH$.
We form the induced subhypergraph $\cH''$ by deleting every vertex $v$ of $\cH'$ satisfying any of the following:
\begin{enumerate}[label=(L\arabic*)]
    \item\label{linear degree} $v$ is the rightmost vertex of more than $(1+\eps/4)dp^{k-1}$ edges of $\cH'$;
    \item\label{linear 3 cycle} $v$ is the rightmost vertex of a $3$-cycle of configuration~\ref{item: 3 cycle type} in $\cH'$; or
    \item\label{linear 4 cycle} $v$ is the rightmost vertex of a $4$-cycle of configuration~\ref{item: 4 cycle type} in $\cH'$.
\end{enumerate}
Since $\cH$ is linear and the ordering is $d$-degenerate, each vertex is the rightmost vertex of at most $k^3d^2$ cycles of configuration~\ref{item: 3 cycle type} and at most $k^3d^3$ cycles of configuration~\ref{item: 4 cycle type}.
Indeed, for the first bound, choose an edge ending at $v$, an edge ending at one of its vertices, and two distinct intersection vertices that determine the remaining edge by linearity. For the second, choose the two edges ending at $v$, one further edge ending at a vertex of the first, and two distinct intersection vertices that determine the last edge.

A $3$-cycle in a linear $k$-uniform hypergraph has $3k-3$ vertices. A $4$-cycle whose nonconsecutive edges are disjoint has $4k-4$ vertices. If a cycle of configuration~\ref{item: 4 cycle type} has an additional intersection between nonconsecutive edges, it contains a cycle of configuration~\ref{item: 3 cycle type} with the same rightmost vertex. Such a vertex is already deleted under~\ref{linear 3 cycle}.
Conditional on $v\in A$, the edges ending at $v$ survive independently, each with probability $p^{k-1}$. Thus, a Chernoff bound and a union bound give
\begin{align*}
\pr\left(v \notin V(\cH'') \,|\, v \in V(\cH')\right)
&\le \underbrace{\exp(-\Omega_{\eps}(dp^{k-1}))}_{\ref{linear degree}}
+ \underbrace{k^3d^2p^{3k-4}}_{\ref{linear 3 cycle}}
+ \underbrace{k^3d^3p^{4k-5}}_{\ref{linear 4 cycle}} \le \frac{\eps}{4},
\end{align*}
for sufficiently large $d$. Indeed, $dp^{k-1}=d^{\mu}\to\infty$, and both cycle terms tend to zero because $0<\mu<(k-2)/(4k-5)$.

Set $d' \coloneqq (1+\eps/4)dp^{k-1}$. Deterministically, $\cH''$ is a $\lfloor d'\rfloor$-degenerate linear $k$-uniform hypergraph, and its inherited ordering contains no cycles of configuration~\ref{item: 3 cycle type} or~\ref{item: 4 cycle type}.
We apply the extension of the sampling argument in the proof of Theorem~\ref{thm: degenerate}\ref{item: degenerate ub} described above to $\cH''$, with error parameter $\eps/4$. The resulting independent set $I$ is also independent in $\cH$, since $\cH''$ is induced.
For sufficiently large $d$, rounding $d'$ down only improves the bound below, since $(\log t)/t$ is decreasing for $t>e$. We obtain
\begin{align*}
    \pr(v \in I) &\ge \pr(v \in V(\cH'')) \, \frac{1}{1+\eps/4} \left(\frac{1}{k-1} \, \frac{\log d'}{d'}\right)^{\frac{1}{k-1}} \\
    &= \pr(v \in V(\cH')) \, \pr\left(v \in V(\cH'') \,|\, v \in V(\cH')\right) \, \frac{1}{1+\eps/4} \left(\frac{1}{k-1} \, \frac{\log d'}{d'}\right)^{\frac{1}{k-1}} \\
    &\ge \left(1-\frac{\eps}{4}\right) \, p \, \frac{1}{1+\eps/4} \left(\frac{1}{k-1} \, \frac{\log d'}{d'}\right)^{\frac{1}{k-1}}.
\end{align*}
Substituting $d'=(1+\eps/4)dp^{k-1}$ gives
\begin{align*}
    \pr(v \in I)
    &\ge \left(1-\frac{\eps}{4}\right) \, p \, \frac{1}{1+\eps/4} \left(\frac{1}{k-1}\, \frac{\log((1+\eps/4)dp^{k-1})}{(1+\eps/4)dp^{k-1}}\right)^{\frac{1}{k-1}} \\
    &\ge \frac{1-\eps/4}{(1+\eps/4)^{1+1/(k-1)}} \left(\frac{\mu}{k-1}\,\frac{\log d}{d}\right)^{\frac{1}{k-1}} \\
    &= \frac{(1-\eps/4)^2}{(1+\eps/4)^{1+1/(k-1)}} \left(\frac{k-2}{(4k-5)(k-1)}\, \frac{\log d}{d}\right)^{\frac{1}{k-1}} \\
    &\ge (1-\eps) \left(\frac{k-2}{(4k-5)(k-1)}\,\frac{\log d}{d}\right)^{\frac{1}{k-1}}.
\end{align*}
For the last inequality, we use $(1+\eps/4)^{-1-1/(k-1)}\ge (1-\eps/4)^2$ and $(1-\eps/4)^4\ge 1-\eps$.
This proves the stated inclusion bound. Applying the construction with $\eps/(1+\eps)$ in place of $\eps$ gives an inclusion probability of at least
\[
\frac{1}{1+\eps}\left(\frac{k-2}{(4k-5)(k-1)}\,\frac{\log d}{d}\right)^{\frac{1}{k-1}}
\]
at every vertex, which yields the claimed upper bound on $\chi_f(\cH)$ and also implies the stated inclusion bound. All deletions and the independent-set sampler can be implemented in polynomial time.

\end{document}